\documentclass[reqno]{amsart}

\usepackage{amsmath}
\usepackage{amsthm}
\usepackage{amscd}
\usepackage{bbm}
\usepackage{enumerate}
\usepackage{amssymb}
\usepackage{palatino}
\usepackage{amscd}
\usepackage{verbatim}
\usepackage{mathrsfs}
\usepackage[margin=1.35in]{geometry}
\usepackage[T2A,T1]{fontenc}
\usepackage[utf8]{inputenc}
\usepackage{manfnt}
\usepackage{url}

\usepackage{color}

\numberwithin{equation}{section}



\title{The average number of rational points on genus two curves is bounded}
\author{Levent Alpoge}\email{lalpoge@math.princeton.edu}
\address{Department of Mathematics, Princeton University, Princeton, NJ 08540.}

\begin{document}

\begin{abstract}
We prove that, when genus two curves $C/\mathbb{Q}$ with a marked Weierstass point are ordered by height, the average number of rational points $\#|C(\mathbb{Q})|$ is bounded. The argument follows the same ideas as the sphere-packing proof of boundedness of the average number of integral points on (quasiminimal Weierstrass models of) elliptic curves. That is, we bound the number of small-height points by hand, the number of medium-height points by establishing an explicit Mumford gap principle and using the theorem of Kabatiansky-Levenshtein on spherical codes (this technique goes back to work of Silverman, Helfgott, and Helfgott-Venkatesh), and the number of large-height points by using Bombieri-Vojta's proof of Faltings' theorem.

Explicitly, in dealing with non-small-height points we prove that the number of rational points $(x,y)$ on $C_f: y^2 = f(x)$ satisfying $h(x) > 8 h(f)$ is $\ll 1.872^{\mathrm{rank}(\mathrm{Jac}(C)(\mathbb{Q}))}$, which has finite average by the theorem of Bhargava-Gross on the average size of $2$-Selmer groups of Jacobians over this family. We note that our arguments in the small-height and large-height cases extend to general genera $g\geq 2$, though for medium points we need to use Stoll's bounds on the non-Archimedean local height differences in genus $2$. For example, we prove that the number of rational points $P\in C(\mathbb{Q})$ with $h(P)\gg_g h(C)$ on $C/\mathbb{Q}$ smooth projective and of genus $g\geq 2$ is $\ll 1.872^{\mathrm{rank}(\mathrm{Jac}(C)(\mathbb{Q}))}$, and that in fact the base of the exponent can be reduced to $1.311$ once $g\gg 1$, though this is surely known to experts (the difference is the use of the Kabatiansky-Levenshtein bound in lieu of more elementary techniques).
\end{abstract}

\maketitle

\newtheoremstyle{dotless}{}{}{\itshape}{}{\bfseries}{}{ }{}

\newtheorem{thm}{Theorem}
\newtheorem{lem}[thm]{Lemma}
\newtheorem{remark}[thm]{Remark}
\newtheorem{cor}[thm]{Corollary}
\newtheorem{defn}[thm]{Definition}
\newtheorem{prop}[thm]{Proposition}
\newtheorem{conj}[thm]{Conjecture}
\newtheorem{claim}[thm]{Claim}
\newtheorem{exer}[thm]{Exercise}
\newtheorem{fact}[thm]{Fact}

\theoremstyle{dotless}

\newtheorem{thmnodot}[thm]{Theorem}
\newtheorem{lemnodot}[thm]{Lemma}
\newtheorem{cornodot}[thm]{Corollary}

\newcommand{\image}{\mathop{\text{image}}}
\newcommand{\End}{\mathop{\text{End}}}
\newcommand{\Hom}{\mathop{\text{Hom}}}
\newcommand{\Sum}{\displaystyle\sum\limits}
\newcommand{\Prod}{\displaystyle\prod\limits}
\newcommand{\Tr}{\mathop{\mathrm{Tr}}}
\renewcommand{\Re}{\operatorname{\mathfrak{Re}}}
\renewcommand{\Im}{\operatorname{\mathfrak{Im}}}
\newcommand{\im}{\mathrm{im}\,}
\newcommand{\inner}[1]{\langle #1 \rangle}
\newcommand{\pair}[2]{\left\langle #1, #2\right\rangle}
\newcommand{\ppair}[2]{\langle\langle #1, #2\rangle\rangle}
\newcommand{\Pair}[2]{\left[#1, #2\right]}
\newcommand{\Char}{\mathop{\mathrm{char}}}
\newcommand{\rank}{\mathrm{rank}}
\newcommand{\sgn}[1]{\mathop{\mathrm{sgn}}(#1)}
\newcommand{\leg}[2]{\left(\frac{#1}{#2}\right)}
\newcommand{\Sym}{\mathrm{Sym}}
\newcommand{\hmat}[2]{\left(\begin{array}{cc} #1 & #2\\ -\bar{#2} & \bar{#1}\end{array}\right)}
\newcommand{\HMat}[2]{\left(\begin{array}{cc} #1 & #2\\ -\overline{#2} & \overline{#1}\end{array}\right)}
\newcommand{\Sin}[1]{\sin{\left(#1\right)}}
\newcommand{\Cos}[1]{\cos{\left(#1\right)}}
\newcommand{\comm}[2]{\left[#1, #2\right]}
\newcommand{\Isom}{\mathop{\mathrm{Isom}}}
\newcommand{\Map}{\mathop{\mathrm{Map}}}
\newcommand{\Bij}{\mathop{\mathrm{Bij}}}
\newcommand{\Z}{\mathbb{Z}}
\newcommand{\R}{\mathbb{R}}
\newcommand{\Q}{\mathbb{Q}}
\renewcommand{\C}{\mathbb{C}}
\newcommand{\Nm}{\mathrm{Nm}}
\newcommand{\RI}[1]{\mathcal{O}_{#1}}
\newcommand{\F}{\mathbb{F}}
\renewcommand{\Pr}{\displaystyle\mathop{\mathrm{Pr}}\limits}
\newcommand{\E}{\mathbb{E}}
\newcommand{\coker}{\mathop{\mathrm{coker}}}
\newcommand{\id}{\mathop{\mathrm{id}}}
\newcommand{\Oplus}{\displaystyle\bigoplus\limits}
\renewcommand{\Cap}{\displaystyle\bigcap\limits}
\renewcommand{\Cup}{\displaystyle\bigcup\limits}
\newcommand{\Bil}{\mathop{\mathrm{Bil}}}
\newcommand{\N}{\mathbb{N}}
\newcommand{\Aut}{\mathop{\mathrm{Aut}}}
\newcommand{\ord}{\mathop{\mathrm{ord}}}
\newcommand{\ch}{\mathop{\mathrm{char}}}
\newcommand{\minpoly}{\mathop{\mathrm{minpoly}}}
\newcommand{\Spec}{\mathop{\mathrm{Spec}}}
\newcommand{\Gal}{\mathop{\mathrm{Gal}}}
\newcommand{\Ad}{\mathop{\mathrm{Ad}}}
\newcommand{\Stab}{\mathop{\mathrm{Stab}}}
\newcommand{\Norm}{\mathop{\mathrm{Norm}}}
\newcommand{\Orb}{\mathop{\mathrm{Orb}}}
\newcommand{\pfrak}{\mathfrak{p}}
\newcommand{\qfrak}{\mathfrak{q}}
\newcommand{\mfrak}{\mathfrak{m}}
\newcommand{\Frac}{\mathop{\mathrm{Frac}}}
\newcommand{\Loc}{\mathop{\mathrm{Loc}}}
\newcommand{\Sat}{\mathop{\mathrm{Sat}}}
\newcommand{\inj}{\hookrightarrow}
\newcommand{\surj}{\twoheadrightarrow}
\newcommand{\bij}{\leftrightarrow}
\newcommand{\Ind}{\mathrm{Ind}}
\newcommand{\Supp}{\mathop{\mathrm{Supp}}}
\newcommand{\Ass}{\mathop{\mathrm{Ass}}}
\newcommand{\Ann}{\mathop{\mathrm{Ann}}}
\newcommand{\Krulldim}{\dim_{\mathrm{Kr}}}
\newcommand{\Avg}{\mathop{\mathrm{Avg}}}
\newcommand{\innerhom}{\underline{\Hom}}
\newcommand{\triv}{\mathop{\mathrm{triv}}}
\newcommand{\Res}{\mathrm{Res}}
\newcommand{\eval}{\mathop{\mathrm{eval}}}
\newcommand{\MC}{\mathop{\mathrm{MC}}}
\newcommand{\Fun}{\mathop{\mathrm{Fun}}}
\newcommand{\InvFun}{\mathop{\mathrm{InvFun}}}
\renewcommand{\ch}{\mathop{\mathrm{ch}}}
\newcommand{\irrep}{\mathop{\mathrm{Irr}}}
\newcommand{\len}{\mathop{\mathrm{len}}}
\newcommand{\SL}{\mathrm{SL}}
\newcommand{\GL}{\mathrm{GL}}
\newcommand{\PSL}{\mathrm{SL}}
\newcommand{\actson}{\curvearrowright}
\renewcommand{\H}{\mathbb{H}}
\newcommand{\mat}[4]{\left(\begin{array}{cc} #1 & #2\\ #3 & #4\end{array}\right)}
\newcommand{\interior}{\mathop{\mathrm{int}}}
\newcommand{\floor}[1]{\left\lfloor #1\right\rfloor}
\newcommand{\iso}{\cong}
\newcommand{\eps}{\epsilon}
\newcommand{\disc}{\mathrm{disc}}
\newcommand{\Frob}{\mathrm{Frob}}
\newcommand{\charpoly}{\mathrm{charpoly}}
\newcommand{\afrak}{\mathfrak{a}}
\newcommand{\cfrak}{\mathfrak{c}}
\newcommand{\codim}{\mathrm{codim}}
\newcommand{\ffrak}{\mathfrak{f}}
\newcommand{\Pfrak}{\mathfrak{P}}
\newcommand{\homcont}{\hom_{\mathrm{cont}}}
\newcommand{\vol}{\mathrm{vol}}
\newcommand{\ofrak}{\mathfrak{o}}
\newcommand{\A}{\mathbb{A}}
\newcommand{\I}{\mathbb{I}}
\newcommand{\invlim}{\varprojlim}
\newcommand{\dirlim}{\varinjlim}
\renewcommand{\ch}{\mathrm{char}}
\newcommand{\artin}[2]{\left(\frac{#1}{#2}\right)}
\newcommand{\Qfrak}{\mathfrak{Q}}
\newcommand{\ur}[1]{#1^{\mathrm{ur}}}
\newcommand{\absnm}{\mathcal{N}}
\newcommand{\ab}[1]{#1^{\mathrm{ab}}}
\newcommand{\G}{\mathbb{G}}
\newcommand{\dfrak}{\mathfrak{d}}
\newcommand{\Bfrak}{\mathfrak{B}}
\renewcommand{\sgn}{\mathrm{sgn}}
\newcommand{\disjcup}{\bigsqcup}
\newcommand{\zfrak}{\mathfrak{z}}
\renewcommand{\Tr}{\mathrm{Tr}}
\newcommand{\reg}{\mathrm{reg}}
\newcommand{\subgrp}{\leq}
\newcommand{\normal}{\vartriangleleft}
\newcommand{\Dfrak}{\mathfrak{D}}
\newcommand{\nvert}{\nmid}
\newcommand{\K}{\mathbb{K}}
\newcommand{\pt}{\mathrm{pt}}
\newcommand{\RP}{\mathbb{RP}}
\newcommand{\CP}{\mathbb{CP}}
\newcommand{\rk}{\mathrm{rk}}
\newcommand{\redH}{\tilde{H}}
\renewcommand{\H}{\tilde{H}}
\newcommand{\Cyl}{\mathrm{Cyl}}
\newcommand{\T}{\mathbb{T}}
\newcommand{\Ab}{\mathrm{Ab}}
\newcommand{\Vect}{\mathrm{Vect}}
\newcommand{\Top}{\mathrm{Top}}
\newcommand{\Nat}{\mathrm{Nat}}
\newcommand{\inc}{\mathrm{inc}}
\newcommand{\Tor}{\mathrm{Tor}}
\newcommand{\Ext}{\mathrm{Ext}}
\newcommand{\fungrpd}{\pi_{\leq 1}}
\newcommand{\slot}{\mbox{---}}
\newcommand{\funct}{\mathcal}
\newcommand{\Funct}{\mathcal{F}}
\newcommand{\Gunct}{\mathcal{G}}
\newcommand{\FunCat}{\mathrm{Funct}}
\newcommand{\Rep}{\mathrm{Rep}}
\newcommand{\Specm}{\mathrm{Specm}}
\newcommand{\ev}{\mathrm{ev}}
\newcommand{\frpt}[1]{\{#1\}}
\newcommand{\h}{\mathscr{H}}
\newcommand{\poly}{\mathrm{poly}}
\newcommand{\Partial}[1]{\frac{\partial}{\partial #1}}
\newcommand{\Cont}{\mathrm{Cont}}
\renewcommand{\o}{\ofrak}
\newcommand{\bfrak}{\mathfrak{b}}
\newcommand{\Cl}{\mathrm{Cl}}
\newcommand{\ceil}[1]{\lceil #1\rceil}
\newcommand{\hfrak}{\mathfrak{h}}
\newcommand{\Sel}{\mathrm{Sel}}
\newcommand{\Qbar}{\overline{\mathbb{Q}}}
\renewcommand{\I}{\mathrm{I}}
\newcommand{\II}{\mathrm{II}}
\newcommand{\III}{\mathrm{III}}
\newcommand{\IV}{\mathrm{IV}}
\newcommand{\V}{\mathrm{V}}
\newcommand{\FuniversalT}{\mathcal{F}_{\mathrm{universal}}^{\leq T}}
\newcommand{\FAT}{\mathcal{F}_{A=0}^{\leq T}}
\newcommand{\FBT}{\mathcal{F}_{B=0}^{\leq T}}
\newcommand{\FcongT}{\mathcal{F}_{\mathrm{congruent}}^{\leq T}}
\newcommand{\rad}{\mathrm{rad}}
\newcommand{\const}{\mathrm{const}}
\renewcommand{\sp}{\mathrm{span}}
\renewcommand{\d}{\partial}
\newcommand{\num}{\mathrm{num}}
\newcommand{\den}{\mathrm{den}}
\newcommand{\ind}{\mathrm{ind}}
\newcommand{\sign}{\mathrm{sign}}
\newcommand{\SU}{\mathrm{SU}}
\renewcommand{\U}{\mathrm{U}}
\newcommand{\tr}{\mathrm{tr}}
\newcommand{\diag}{\mathrm{diag}}
\renewcommand{\O}{\mathcal{O}}
\renewcommand{\P}{\mathbb{P}}
\newcommand{\Jac}{\mathop{\mathrm{Jac}}}
\newcommand{\FgoodT}{\mathcal{F}_{\mathrm{good}}^{\leq T}}
\newcommand{\FbadT}{\mathcal{F}_{\mathrm{bad}}^{\leq T}}
\newcommand{\up}{\uparrow}
\newcommand{\down}{\downarrow}
\renewcommand{\div}{\mathrm{div}}
\newcommand{\Pic}{\mathrm{Pic}}
\newcommand{\Funiv}{\mathcal{F}_{\mathrm{universal}}}
\newcommand{\FunivT}{\Funiv^{\leq T}}
\newcommand{\Prob}{\mathrm{Prob}}
\newcommand{\Zh}{\mbox{\usefont{T2A}{\rmdefault}{m}{n}\CYRZH}}
\newcommand{\zh}{\mbox{\usefont{T2A}{\rmdefault}{m}{n}\cyrzh}}
\newcommand{\da}{\mbox{\usefont{T2A}{\rmdefault}{m}{n}\cyrd}}
\newcommand{\columnvector}[2]{\left(\begin{array}{cc} #1 \\ #2\end{array}\right)}
\newcommand*\cube{\mbox{\mancube}}

\let\uglyphi\phi
\let\phi\varphi

\setcounter{tocdepth}{4}
\setcounter{secnumdepth}{4}

\tableofcontents

\section{Introduction}
In Book VI of his \emph{Arithmetica}, which, incidentally, was only discovered in 1968 (in Arabic translation in a shrine in Iran by Fuat Sezgin), Diophantus \cite{diophantus} poses the problem of finding a nontrivial rational point on the curve $C: y^2 = x^6 + x^2 + 1$. He finds the point $\left(\frac{1}{2}, \frac{9}{8}\right)$. This gives eight 'obvious' points: $\left(\pm \frac{1}{2}, \pm \frac{9}{8}\right)$, $(0, \pm 1)$, and the two points at infinity. One is of course led to ask whether there are any nonobvious rational points on this curve. In 1998, some 1700 years after Diophantus, Wetherell (in his PhD dissertation \cite{wetherell}) answered this question: Diophantus had in fact found all the rational points on $C$! The recency of this result should serve as an indication of our relative ignorance of the arithmetic of genus $2$ curves as compared to that of elliptic curves.

In this paper we study rational points on genus $2$ curves, on average. Faltings' theorem tells us that the set of rational points on such a (smooth projective) curve is finite. How finite?

Let $f\in \Z[x]$ be a monic quintic polynomial of the form $f(x) =: x^5 + a_2 x^3 + \cdots + a_5$ with nonvanishing discriminant $\Delta_f\neq 0$. Let $H(f) := \max_i |a_i|^{\frac{1}{i}}$ be the na{\" i}ve height of $f$. Let $C_f: y^2 = f(x)$. This gives a family $\Funiv$ of genus $2$ curves over $\Q$ with a marked Weierstrass point (the point at infinity), which we order by height.\footnote{Conversely, given a genus $2$ curve $C$ with a marked Weierstrass point $P\in C(\Q)$, via Riemann-Roch, there is a nonconstant $f\in \O(2P)$, which realizes $C$ as a a degree-two ramified covering of $\P^1$, ramified (by Riemann-Hurwitz) over six points (including $\infty$, the image of $P$). Such a thing is precisely a quintic polynomial, and, on translating and scaling, we get a quintic of the considered form --- even with integral coefficients since the ramification set is Galois-invariant (well, for a suitable choice of $f$).}\textsuperscript{,}\footnote{We note here that we expect all the arguments of the paper go through without much trouble for the family $y^2 = x^6 + a_2 x^4 + \cdots + a_6$ of genus two curves with a marked rational \emph{non}-Weierstrass point. We have not dealt with this case here due to the lack of a compelling reason to face the significantly more complicated formulas that arise --- granted, the complication is at least balanced by much better lower bounds on the height of a non-small point in the Kummer surface. The general case, $y^2 = a_0 x^6 + \cdots + a_6$, does not have a corresponding $2$-Selmer average bound for the Jacobian itself, and anyway even the height on the Kummer surface would be extremely complicated to deal with, since one would have to work with the Abel-Jacobi map associated to a rational point potentially not at $\infty$. (One could question whether this ordering of curves is even worth dealing with once one has the monic sextic case in hand.)} Given a family of curves $\mathcal{F}$ and a function $f$ on this family, we write $$\Avg_{\Funct^{\leq T}}(f) := \frac{\sum_{C\in \Funct : H(C)\leq T} f(C)}{\sum_{C\in \Funct : H(C)\leq T} 1}.$$ By a statement like $\Avg_{\Funct}(f)\leq B$ we will mean $\limsup_{T\to\infty} \Avg_{\Funct^{\leq T}}(f)\leq B$. Thus for example Bhargava-Gross \cite{bhargavagross} have proved that $\Avg_{\Funiv}(2^{\rank(\Jac{C})})\leq 3$.

Now recall that Faltings tells us that $\#|C_f(\Q)| < \infty$ for all $C_f\in \Funiv$. It is a famous open problem to determine whether or not these point counts are uniformly bounded in $L^\infty$. We make no progress on this question, but we do prove the much weaker statement that the point counts are uniformly bounded in $L^1$. (In fact the argument gives bounds in $L^p$ for $p$ slightly larger than $1$ as well.)

\begin{thm}\label{the big kahuna}
$$\Avg_{\Funiv}(\#|C(\Q)|) < \infty.$$
\end{thm}

We are certain that the limiting distribution of $\#|C(\Q)|$ should just be $\delta_1$ --- i.e., the probability of having anything more than the point at infinity should be $0$. Thus we expect the average to be $1$, but proving this is certainly far out of our reach.

\section{Acknowledgments}
I would like to thank Manjul Bhargava, Peter Sarnak, Michael Stoll, Jacob Tsimerman, and especially Steffen M{\"u}ller for enlightening conversations. This research was supported by the NSF GRFP.

\section{Notation, previous results, and outline of the argument}

\subsection{Notation.}
We write $f\ll_\theta g$ to mean that there is some $C_\theta > 0$ depending only on $\theta$ such that $|f|\leq C_\theta |g|$ pointwise. Thus $\ll$ and $\gg$, sans decoration, imply the same inequalities with absolute constants. By $f\asymp_\theta g$ we will mean $g\ll_\theta f\ll_\theta g$. By $O_\theta(g)$ we will mean a quantity that is $\ll_\theta g$, and by $\Omega_\theta(g)$ we will mean a quantity that is $\gg_\theta g$. By $o(1)$ we will mean a quantity that approaches $0$ in the relevant limit (which will always be unambiguous). By $f = o(g)$ we will mean $f = o(1)\cdot g$. We will write $(a,b)$ for the greatest common divisor of the integers $a,b\in \Z$, $\omega(n)$ for the number of prime factors of $n\in \Z$, $v_p$ for the $p$-adic valuation, $|\cdot|_v$ for the absolute value at a place $v$ of a number field $K$ (normalized so that the product formula holds), and $h(x)$ for the absolute Weil height of $x\in \Qbar$ --- i.e., if $x\in K$, $$h(x) := \sum_w \frac{[K_w : \Q_v]}{[K : \Q]} \log^+{|x|_w},$$ the sum taken over places $w$ of the number field $K$, with $v := w\vert_{\Q}$ the place of $\Q$ over which $w$ lies and $\log^+(a) := \max(0, \log{a})$.\footnote{To be explicit about our normalization, as usual $|p|_p := p^{-1}$ and we will take $|\cdot|_\infty$ to be the standard absolute value, and then $|\cdot|_w := |\Nm_{L_w/\K_v}(\cdot)|_v^{\frac{1}{[L_w:K_v]}}$ for $L/K$ an extension of number fields and $v := w\vert_K$.} We will also write $H(x) := \exp(h(x))$ for the multiplicative Weil height, so that $H\left(\frac{a}{b}\right) = \max(|a|,|b|)$ if $\frac{a}{b}\in \Q$ is in lowest terms. For a rational point $P\in C(\Q)$, we will write $h(P)$ and $H(P)$ for $h(x(P))$ and $H(x(P))$, respectively. We will write $J_f := \Jac(C_f)$ for the Jacobian of the curve $C_f$, with embedding $C_f\inj J_f$ determined by the Abel-Jacobi map associated to the point at infinity (whence $J_f$ and the map $C_f\inj J_f$ are both defined over $\Q$). We will also write $K_f := J_f/\{\pm 1\}$ for the associated \emph{Kummer variety}, the key player in this paper. We will use an explicit embedding $K_f\inj \P^3$ of Cassels-Flynn \cite{casselsflynn} to furnish a height $h_K$ on $K_f(\Qbar)$ by pulling back the standard logarithmic Weil height on $\P^3$ via the embedding. (The height used on $\P^n$ is $$h_{\P^n}([x_0, \ldots, x_n]) := \sum_w \frac{[K_w : \Q_v]}{[K : \Q]} \max_{0\leq i\leq n} \log^+{|x_i|_w},$$ which is well-defined by the product formula.) $$\hat{h}(P) := \lim_{n\to\infty} \frac{h_K(2^n P)}{4^n}$$ will denote the canonical height of a point $P\in K_f(\Qbar)$, and we will write $$\lambda_w(P) := \frac{[K_w : \Q_v]}{[K : \Q]} \max_{1\leq i\leq 4}\log^+{|P_i|_w}$$ and $$\hat{\lambda}_w(P) := \lim_{n\to\infty} \frac{\lambda_w(2^n P)}{4^n},$$ so that $h = \sum_w \lambda_w$ and $\hat{h} = \sum_w \hat{\lambda}_w$. Note that, as defined, both $\lambda_w$ and $\hat{\lambda}_w$ depend on the choice of coordinates $P\mapsto P_i$, and thus are only functions on the cone on $K_f$ in $\A^4$ (usually the notation $\hat{\lambda}_w$ is reserved for the local N\'{e}ron functions). By $\Delta$ or $\Delta_f$ we will mean $\Delta_f := 2^8 \disc(f)$, where $\disc(f)$ is the discriminant of $f$ considered as a polynomial in $x$ (i.e., the resultant of $f$ and $f'$). In general logarithmic heights will be in lower case, and their exponentials in upper case (much the same way as our $H = \exp(h)$ on $\Qbar$).

Finally, in general we will use the notation $(\in d\Z)$ to indicate an element of $d\Z$ in an expression --- e.g., $a = \frac{b + (\in c\Z)}{b' + (\in c'\Z)}$ is equivalent to $a = \frac{d}{d'}$ with $d\equiv b\pmod{c}$ and $d'\equiv b'\pmod{c'}$. 

\subsection{Previous results.}
The fundamental finiteness theorem that allows us to even begin asking about $L^1$ averages is Faltings' \cite{faltings}:
\begin{thm}[Faltings]
Let $K$ be a global field and $C/K$ a smooth projective curve over $K$. Then: $$\#|C(K)| < \infty.$$
\end{thm}
Unfortunately this theorem is still (to the knowledge of the author, though it seems Mochizuki's work has bearing on this) ineffective, in the sense that it does not give a bound on the heights of rational points on such curves. The analogous situation for integral points on elliptic curves, where the relevant finiteness theorem is Siegel's, differs in this respect due to work of Baker. In any case, Faltings' proof(s), as well as Vojta's and Bombieri's (following Vojta's), all give bounds on the \emph{number} of rational points on these curves. To the author's knowledge, the best general bound is provided by the following theorem, which we state here in the case of curves of the form $y^2 = f(x)$ with $\deg{f}$ odd\footnote{This hypothesis is absolutely not essential. The evident generalization to curves of genus $g\geq 2$ is the one they actually prove --- I have just written the special case of odd-degree hyperelliptic curves for convenience.} (and which I have attributed to Bombieri-Vojta, with the determination of the dependence of the implicit constants on the curve done by Bombieri-Granville-Pintz):
\begin{thm}[Bombieri-Vojta, Bombieri-Granville-Pintz]
Let $\delta > 0$. Let $f\in \Z[x]$ be a monic polynomial of degree $2g+1\geq 5$ with no repeated roots, and let $C_f: y^2 = f(x)$. Let $\alpha\in \left(\frac{1}{\sqrt{g}}, 1\right)$. Let $P\neq Q\in C_f(\Q)$ be such that $\hat{h}(P)\geq \delta^{-1}\hat{h}(Q)\geq \delta^{-2} h(f)$. Then, once $\delta\ll_\alpha 1$, we have that: $$\cos{\theta_{P,Q}}\leq \alpha.$$
\end{thm}

That this does in fact imply Faltings' theorem is quickly seen on remarking that this forces rational points to either be of bounded height, or in a finite (finite because $C_f(\Q)\otimes_\Z \R$ is finite-dimensional, by Mordell-Weil) number of cones emanating from the origin. Within each such cone, all points must again be of bounded height, QED. Note that this still does not give a bound on the heights of the points, since the height bound within each cone depends on a chosen rational point in the cone (if one exists --- if one doesn't we are also done!), and a priori one does not know any upper bound on its height.

In any case, this theorem does imply a very good bound on the number of large-height points on $C_f$, via the Kabatiansky-Levenshtein bound and the Mumford gap principle. Specifically, we obtain that the number of large-height points is $\ll 1.872^{\rank(J_f(\Q))}$, where $J_f$ is the Jacobian of $C_f$ (with Abel-Jacobi map corresponding to the point at infinity) --- see Lemma \ref{the large point bound}.\footnote{In fact we get a stronger bound than this, for all genera --- see Proposition \ref{the optimized general large point bound}.} This is the observation that makes plausible an attempt to prove $L^1$ boundedness given the current state of the art in average Selmer bounds.

Indeed, this bound on the number of large points is enough for us because of the following theorem of Bhargava-Gross \cite{bhargavagross} (combined with $2^{\rank(J_f(\Q))}\leq \#|\Sel_2(J_f)|$):
\begin{thm}[Bhargava-Gross]\label{bhargava gross}
$$\Avg_{\Funiv} \#|\Sel_2(\Jac(C))| = 3.$$
\end{thm}

In fact Bhargava-Gross \cite{bhargavagross} (for odd degrees) and Shankar-Wang \cite{shankarwang} (for even degrees) proved the same theorem (for even degrees the average is $6$, since now there are two points at infinity) for \emph{any} $d$ --- that is, for the family $\Funct_{\mathrm{universal}, \deg = d}$ of fixed degree $d\geq 5$ \emph{monic} polynomials $f(x) = x^d + a_2 x^{d-2} + \cdots\in \Z[x]$, again ordered by height (with $H(f) := \max_i |a_i|^{\frac{1}{i}}$, as above). Since the average is bounded independently of the genus, it follows that the proportion of hyperelliptic curves whose Jacobian has Mordell-Weil rank $\geq g$ is exponentially small in $g$. Thus the method of Chabauty-Coleman applies to all but an exponentially small proportion of these curves! Pushing this much, much further, Poonen-Stoll \cite{poonenstoll} (for odd degrees) and Shankar-Wang \cite{shankarwang} (for even degrees) proved:\footnote{Here, as with the use of $\Avg$ in Theorem \ref{the big kahuna}, there is an implicit $\limsup$ in the statement.}
\begin{thm}[Poonen-Stoll, Shankar-Wang]
Let $g\geq 2$ and $d = 2g+1$ or $2g+2$. Then:
$$\Prob_{\Funct_{\mathrm{universal}, \deg = d}}\left(C(\Q)\neq \{\infty^{\pm}\}\right)\ll g 2^{-g}.$$
\end{thm}
Thus we have very strong bounds in $L^0$ --- improving exponentially quickly with the genus --- in all degrees.

Let us now outline the argument.

\subsection{Outline of the argument.}

Throughout $\delta > 0$ will be a (small, but still $\gg 1$) parameter, eventually taken to be $\ll 1$ --- e.g., $\delta = 10^{-10^{10}}$ will work (certainly nothing remotely this small is actually necessary, of course).

Just as in the case of integral points on elliptic curves, we break rational points on these curves into three types: \emph{small}, \emph{medium}, and \emph{large} points. By ``large'', we will mean that a point satisfies the lower bound of the Bombieri-Granville-Pintz refinement of Bombieri-Vojta \cite{bombierivojta} --- that is, for $P\in C_f(\Q)$, it will mean that $h(P) > \delta^{-1} h(f)$. The large points are bounded by an appeal to the Bombieri-Vojta proof of Faltings, which eventually will give us a bound of $\ll 1.872^{\rank(J_f(\Q))}$. We note for now and for the rest of the paper that any bound below $2^{\rank(J_f(\Q))}$ is enough to imply the theorem, since $2^{\rank(J_f(\Q))}\leq \#|\Sel_2(J_f)|$, and the latter has bounded average, by Bhargava-Gross.

Thus the problem is immediately reduced to the small and medium points. Next easiest to handle are the small points, which, in much the same way as the case of integral points on elliptic curves, we handle by switching the order of summation --- i.e. (roughly), counting how many \emph{curves} a given point can lie on, and then summing over possible points. Here we use a technique which might be called 'attraction' --- for example, if $x^5 + a$ and $x^5 + b$ are squares, say $y^2$ and $y'^2$, respectively (with $y\neq y' > 0$), and $x\in \Z$ satisfies $|x|\geq \delta^{-1} \max(|a|,|b|)^{\frac{1}{5}}$, it follows that $1\leq |y - y'| = \frac{|a-b|}{y+y'}\ll x^{\frac{5}{2}}$, whence there are $O(x^{\frac{5}{2}})$ many $c$'s such that $x^5 + c$ is a square and $|c|\ll \delta x^5$. (This is, of course, evident by other means, namely e.g.\ counting the square roots.) This sort of argument, along with some congruences that must be satisfied modulo powers of the denominator of our point (since we must deal with rational points), will be applied coefficient-by-coefficient to get a good bound on small points.

Finally, what remain are the medium points. In the case of integral points on elliptic curves, we used an explicit gap principle due to Helfgott and Silverman, which we described as an analogue of the Mumford gap principle on higher-genus curves. Here we simply use the Mumford gap principle, except of course we must be very precise with the dependence of the error terms on the curve (which is always the difficulty in these sorts of explicit questions).

Thankfully, due to work of Cassels and Flynn, given a curve of the form $C_f$, we get explicit embeddings $J_f\inj \P^{15}$ and $K_f\inj \P^3$, where $J_f := \Jac(C_f)$ as usual and $K_f := J_f/\{\pm 1\}$ again. Moreover we get an explicit multiplication-by-$2$ map $2: K_f\to K_f$ as well as an explicit addition law $J_f\times J_f\to J_f$. Finally the Abel-Jacobi map $C_f\inj J_f$ corresponding to the point at infinity is extremely simple (the composition with $J_f\surj K_f\inj \P^3$ factors through a Veronese embedding $C_f\to \P^1\inj \P^2\subseteq \P^3$). All this makes the explicit analysis of the \emph{na{\" i}ve} height of the sum (and difference) of two points quite possible.

Next we use work of Stoll \cite{stolli,stollii} to transfer the bounds we get to bounds on the \emph{canonical} height. We will need to take a detour into explicitly calculating the canonical local 'height' at $\infty$, and for good enough bounds on this we will need to introduce extra partitioning, but let us ignore this for the sake of the outline. In the end we will balance the upper bound that arises from analyzing the sum of two points (which is good enough in some cases) and that arising from the analysis of the difference (which is good enough in, in fact, all the other cases!) to obtain a sufficiently strong gap principle to provide an admissible bound on the number of medium points. The difficulty here will be that we can only explicitly count small points up to a certain threshold, while the gap principle worsens as the height of the points decreases.

The explicit gap principle will then imply that the number of medium-sized points $P\in C_f(\Q)$ with $\hat{h}(P)\in [A, (1+\delta) A], h(P)\in [B, (1+\delta) B]$, and $B\geq \const\cdot h(f)$ (for $\const$ a suitably chosen constant --- we will see that it will depend on whether or not $x(P)$ is much larger than $H(f)$, though let us ignore this), is at most $\ll 1.645^{\rank(J_f(\Q))}$. Indeed, any two points with heights in such an interval will be forced to 'repulse', i.e.\ to have large angle between them. Then, on projecting the points to the unit sphere in $C_f(\Q)\otimes_\Z \R$, it will follow that we are tasked with upper bounding the size of a spherical code with given minimal angle, which is precisely the content of the Kabatianksy-Levenshtein bound. This will complete the argument.

As a final remark, there is of course the question of generalizing the argument to higher degrees. Our analysis of small and large points carries over immediately to polynomials of arbitrary degree $d\geq 5$ (and, for the large point analysis, even to general genus $g\geq 2$ curves), but where I see no way forward is the analysis of medium points. Specifically, the existence of explicit equations embedding $J_f$ and $K_f$ into projective spaces, as well as explicit formulas for the addition law on $J_f$ and for duplication on $K_f$, \emph{as well} as explicit analysis of local height functions for the canonical height (especially as compared to those for the na{\" i}ve height) is absolutely crucial to the techniques and thus the method seems resistant to easy generalization. It seems the relevant facts about local heights may exist for genus $3$ hyperelliptic curves due to recent work of Stoll \cite{stollgenusthree}, but I cannot say yet whether this technique goes through. For yet higher degrees it seems significantly different (i.e., not so explicit) techniques are necessary.

In any case, this completes the outline of the argument. Our next step will be to collect, for reference, the definition of the final partition used in the argument.

\begin{remark}
The following section is redundant and simply collects the definitions of the various parts of the partition, and so can be skipped and used as a reference (I have included it only because the notation gets quite intricate --- the paper is independent of the section, for example).
\end{remark}

\section{The partition}

\subsection{Small, medium-sized, and large points.}
We begin by defining ``small'', ``medium-sized'', and ``large'' points. It turns out our height bounds will depend on whether or not our points have unusually large $x$-coordinates or not, principally because points with very large $x$-coordinate are very close to $\infty$ in the Archimedean topology, and so are very easy to handle with explicit formulas. In any case, let:
\begin{align*}
\I_f^\up &:= \{P\in C_f(\Q) : |x(P)|\geq \delta^{-\delta^{-1}} H(f), h(P) < (c_\up - \delta) h(f)\},\\
\I_f^\down &:= \{P\in C_f(\Q) : |x(P)| < \delta^{-\delta^{-1}} H(f), h(P) < (c_\down - \delta) h(f)\},\\
\I_f &:= \I_f^\up \cup \I_f^\down,\\
\II_f^\up &:= \{P\in C_f(\Q) : |x(P)|\geq \delta^{-\delta^{-1}} H(f), P\not\in \I_f, h(P) < \delta^{-\delta^{-1}} h(f)\},\\
\II_f^\down &:= \{P\in C_f(\Q) : |x(P)|\leq \delta^{\delta^{-1}} H(f), P\not\in \I_f, h(P) < \delta^{-\delta^{-1}} h(f)\},\\
\II_f^\bullet &:= \{P\in C_f(\Q) : \delta^{\delta^{-1}} H(f) < |x(P)| < \delta^{-\delta^{-1}} H(f), P\not\in \I_f, h(P) < \delta^{-\delta^{-1}} h(f)\},\\
\II_f &:= \II_f^\up\cup \II_f^\bullet\cup \II_f^\down,\\
\III_f &:= C_f(\Q) - (\I_f\cup \II_f),\\
c_\up &:= \frac{25}{3},\\
c_\down &:= 8.
\end{align*}

\subsection{Ensuring that lifts to $\C^2$ via $\C^2\surj J_f(\C)$ are close, and closeness to $\alpha_*$ versus $\beta_*$.}
Next, we will refine this partition for medium-sized and large points. We will first partition points so that two points in the same part have very close lifts to a fundamental domain for the lattice $\Z^2 + \tau_f\cdot \Z^2$ in $\C^2$, where $\tau_f$ is a Riemann matrix of the Jacobian of $y^2 = f(x)$ (in the Siegel fundamental domain), and we have fixed an isomorphism $\Psi_f: \C^2/(\Z^2 + \tau_f\cdot \Z^2)\simeq J_f(\C)$. (See Section \ref{lower bound on h hat of P-Q} for details.)

Write $$\Gunct := \left[-\frac{1}{2}, \frac{1}{2}\right]^{\times 2} + \tau_f\cdot \left[-\frac{1}{2},\frac{1}{2}\right]^{\times 2}.$$ Write $$\Gunct^{(i_1, \ldots, i_4)} := \left(\left[\frac{i_1}{2N}, \frac{i_1+1}{2N}\right]\times \left[\frac{i_2}{2N}, \frac{i_2+1}{2N}\right]\right) + \tau_f\cdot \left(\left[\frac{i_3}{2N}, \frac{i_3+1}{2N}\right]\times \left[\frac{i_4}{2N}, \frac{i_4+1}{2N}\right]\right),$$ where $N\asymp \delta^{-1}$ (thus this is a partition into $O(1)$ parts, since $\delta\gg 1$). Note that $$\Gunct = \bigcup_{i_1 = -N}^N\bigcup_{i_2 = -N}^N\bigcup_{i_3 = -N}^N\bigcup_{i_4 = -N}^N \Gunct^{(i_1,i_2,i_3,i_4)}.$$ Now let $$Z: J_f(\C)\to \Gunct$$ be a set-theoretic section (observe that the map $\Gunct\to \C^2/(\Z^2 + \tau_f\cdot \Z^2)\simeq J_f(\C)$ is surjective). Next let $$\II_f^{(i_1,i_2,i_3,i_4)} := \II_f\cap Z^{-1}(\Gunct^{(i_1,i_2,i_3,i_4)}).$$ (Similarly with decorations such as $\up, \down, \bullet$ added, and for $\III_f^{(i_1, \ldots, i_4)}$.) Thus if $P,Q\in \II_f^{(i_1,i_2,i_3,i_4)}$, we have that $$Z(P)-Z(Q) = A + \tau_f\cdot B$$ with $$||A||, ||B||\ll \delta.$$ Finally, recall (via Lemma \ref{producing alpha star and beta star}) that we may, and will, choose two roots $\alpha_*, \beta_*$ of $f$ such that $$|\alpha_*|, |\beta_*|, |\alpha_* - \beta_*|\gg H(f).$$ So let $$\II_f^{\alpha_*} := \{P\in \II_f : |x(P) - \alpha_*|\leq |x(P) - \beta_*|\}$$ and, similarly, $$\II_f^{\beta_*} := \{P\in \II_f : |x(P) - \beta_*|\leq |x(P) - \alpha_*|\}.$$ That is, $\II_f^{\alpha_*}$ is the set of points of $\II_f$ whose $x$-coordinates are closest to $\alpha_*$, and similarly for $\II_f^{\beta_*}$. We similarly define $\III_f^{\alpha_*}$ and $\III_f^{\beta_*}$. Finally, for $\rho\in \{\alpha_*, \beta_*\}$, define $$\II_f^{(i_1,i_2,i_3,i_4), \rho} := \II_f^{(i_1,i_2,i_3,i_4)}\cap \II_f^\rho,$$ and similarly for $\III_f$, and all other decorations.

\subsection{Ensuring that the canonical and naive heights are multiplicatively very close.}
Finally, we refine the partition once more, in order to ensure that two points in the same part have very (multiplicatively) close canonical \emph{and} naive heights.

Lemma \ref{the heights are comparable} furnishes us with constants $\mu, \nu$ with $\mu\asymp 1, \nu\asymp \delta^{-\delta^{-1}}$ such that, for all $P\in \II_f$, we have that both $\hat{h}(P)\in [\mu^{-1} h_K(P), \mu h_K(P)], h_K(P)\in [\nu^{-1} h(f), \nu h(f)]$.

Note that $$[\mu^{-1}, \mu]\subseteq \Cup_{i=-O(\delta^{-1})}^{O(\delta^{-1})} [(1+\delta)^i, (1+\delta)^{i+1}],$$ and similarly for $[\nu^{-1}, \nu]$ (except with the bounds on the union changed to $O(\delta^{-2}\log{\delta^{-1}})$). Define the following partition of $\II_f$ into $\delta^{-O(1)}$ many pieces: $$\II_f^{[i,j]} := \{P\in \II_f : \hat{h}(P)\in [(1+\delta)^i h_K(P), (1+\delta)^{i+1} h_K(P)]\text{ and } h_K(P)\in [(1+\delta)^j h(f), (1+\delta)^{j+1} h(f)]\},$$ and similarly with all other decorations added --- e.g., $$\II_f^{\up, (i_1, i_2, i_3, i_4), \rho, [i,j]} := \II_f^{\up, (i_1, i_2, i_3, i_4), \rho}\cap \II_f^{[i,j]}.$$ We also define $\III_f^{[[i]]}$, etc.\ (thus also e.g.\ $\III_f^{\bullet, (i_1, \ldots, i_4), \rho, [[i]]}$) in a similar way, except without the second condition --- that is, we only impose that $\hat{h}(P)\in [(1+\delta)^i h_K(P), (1+\delta)^{i+1} h_K(P)]$.

Thus $$\II_f = \bigcup_{\rho\in \{\alpha_*, \beta_*\}}\bigcup_{?\in \{\up,\bullet,\down\}}\bigcup_{i_1 = 0}^{O(\delta^{-1})}\cdots \bigcup_{i_4 = 0}^{O(\delta^{-1})}\bigcup_{i=-O(\delta^{-1})}^{O(\delta^{-1})}\bigcup_{j=-\delta^{-O(1)}}^{\delta^{-O(1)}} \II_f^{?,(i_1,\ldots,i_4),\rho,[i,j]},$$ and, since the partition is into $\delta^{-O(1)} = O(1)$ parts, it suffices to bound the sizes of each part individually. Similarly for $\III_f$.

Now let us prove the theorem.	

\section{Proof of Theorem \ref{the big kahuna}}
The first thing to note is that $|\Funct^{\leq T}_{\mathrm{universal}}|\asymp T^{14}$. The second thing to note is that if $\left(\frac{a}{b}, \frac{c}{d}\right)\in C_f(\Q)$ is in lowest terms, then since (by clearing denominators of the defining equation) $d^2\vert c^2 b^5$ and $b^5\vert d^2 \cdot (a^5 + (\in b\Z))$, it follows that $d^2 = b^5$ --- i.e., that all rational points on $C_f$ are of the form $\left(\frac{a}{e^2}, \frac{b}{e^5}\right)$ in lowest terms.

Now, given $f\in \Funct^{\leq T}_{\mathrm{universal}}$, write $C_f(\Q) =: \I_f \cup \II_f \cup \III_f$, with:\footnote{In the general degree $d$ case, we have that $|\Funct^{\leq T}_{\mathrm{universal}}|\asymp T^{\frac{d(d+1)}{2} - 1}$, that $c_\up = \frac{d(d+1) - 5}{3}$, and that $c_\down = \frac{d(d+1)}{3} - 2$, via precisely the same methods. (These values of $c_\up$ and $c_\down$ are easily improved upon in the general case, though I do not know precisely how far one can go.)}
\begin{align*}
\I_f^\up &:= \{P\in C_f(\Q) : |x(P)|\geq \delta^{-\delta^{-1}} H(f), h(P) < (c_\up - \delta) h(f)\},\\
\I_f^\down &:= \{P\in C_f(\Q) : |x(P)| < \delta^{-\delta^{-1}} H(f), h(P) < (c_\down - \delta) h(f)\},\\
\I_f &:= \I_f^\up \cup \I_f^\down,\\
\II_f^\up &:= \{P\in C_f(\Q) : |x(P)|\geq \delta^{-\delta^{-1}} H(f), P\not\in \I_f, h(P) < \delta^{-\delta^{-1}} h(f)\},\\
\II_f^\down &:= \{P\in C_f(\Q) : |x(P)|\leq \delta^{\delta^{-1}} H(f), P\not\in \I_f, h(P) < \delta^{-\delta^{-1}} h(f)\},\\
\II_f^\bullet &:= \{P\in C_f(\Q) : \delta^{\delta^{-1}} H(f) < |x(P)| < \delta^{-\delta^{-1}} H(f), P\not\in \I_f, h(P) < \delta^{-\delta^{-1}} h(f)\},\\
\II_f &:= \II_f^\up\cup \II_f^\bullet\cup \II_f^\down,\\
\III_f &:= C_f(\Q) - (\I_f\cup \II_f),\\
c_\up &:= \frac{25}{3},\\
c_\down &:= 8.
\end{align*}

We will call points in $\I_f$ \emph{small}, points in $\II_f$ \emph{medium}-sized, and points in $\III_f$ \emph{large}. In words, the decorations $\up, \bullet, \down$ indicate the size of the $x$-coordinates of such points, and we have broken into: large points, and small/medium-sized points with surprisingly large/surprisingly small/otherwise (the second occurring only in the case of medium points) $x$-coordinates. Let us first handle the small points.

\subsection{Small points.}
We first turn to those small points with large $x$-coordinate, i.e.\ the points in $\I_f^\up$.
\begin{lem}\label{big lookin small points}
$$\sum_{f\in \FunivT} \#|\I_f^\up|\ll T^{14 - \Omega(\delta)}.$$
\end{lem}
\begin{proof}
Certainly
\begingroup
\addtolength{\jot}{-0.2in}
\begin{align*}
\sum_{f\in \FunivT, H(f)\asymp T} \#|\I_f^\up|\leq \#|\{&(s,d,a_2,t,a_3,a_4,a_5) : t^2 = s^5 + a_2 d^4s^3 + \cdots + a_5 d^{10}, \\&|a_i|\leq T^i, \Delta_f\neq 0, (s,d) = 1, (t,d) = 1, T\ll |s|\leq T^{c_\up - \delta}, d^2\leq \delta^{\delta^{-1}} |s| T^{-1}\}|,
\end{align*}
\endgroup
and it suffices to prove the claimed bound for this sum (via a dyadic partition of size $\log{T}$ --- note that we have restricted $H(f)\asymp T$, rather than just $\leq T$).

We will say that an integer $z$ \emph{extends} the tuple $(\tilde{s}, \tilde{d}, \ldots)$ if there is an element $(s,d,a_2,t,a_3,a_4,a_5)$ in the above set agreeing in the respective entries --- i.e., $s = \tilde{s}, d = \tilde{d}, \ldots$ --- and with value $z$ in the next entry. (The tuple can be length one, or even empty, of course.) For example, given an element $(s,d,a_2,t,a_3,a_4,a_5)$ of the above set, $a_5$ extends $(s,d,a_2,t,a_3,a_4)$, $a_4$ extends $(s,d,a_2,t,a_3)$, and so on.

Given a tuple $(s,d,a_2,t,a_3,a_4)$, of course there is at most one $a_5$ that extends this tuple --- indeed, the defining equation lets us solve for $a_5 d^{10}$, and hence for $a_5$. Next, given $(s,d,a_2,t,a_3)$, if both $a_4$ and $a_4'$ extend $(s, \ldots, a_3)$ to $(s, \ldots, a_3, a_4, a_5)$ and $(s, \ldots, a_3, a_4', a_5')$, respectively, then, on taking differences of the respective defining equations, we find that $$0 = (a_4 - a_4') d^8 s + (a_5 - a_5') d^{10},$$ and so $$|a_4 - a_4'|\ll \frac{T^5 d^2}{|s|}$$ since $|a_5|, |a_5'|\ll T^5$. Moreover we also have that $0\equiv (a_4 - a_4') s d^8\pmod{d^{10}}$, i.e.\ that (since $(s,d) = 1$) $$a_4\equiv a_4'\pmod{d^2}.$$ Hence the number of $a_4$ extending the tuple $(s,d,a_2,t,a_3)$ is $$\ll 1 + \frac{T^5}{|s|}.$$ Similarly, if $a_3$ and $a_3'$ extend $(s,d,a_2,t)$, then, with similar notation as before, $$0 = (a_3 - a_3') s^2 d^6 + (a_4 - a_4') s d^8 + (a_5 - a_5') d^{10},$$ whence (here we use that $|s|\gg T d^2$) $$|a_3 - a_3'|\ll \frac{T^4 d^2}{|s|},$$ and, again, on reducing the equation modulo $d^8$ it follows that $$a_3\equiv a_3'\pmod{d^2}.$$ Thus the number of $a_3$ extending $(s,d,a_2,t)$ is $$\ll 1 + \frac{T^4}{|s|}.$$

Next, if both $t, t' > 0$ extend $(s,d,a_2)$, then $$(t - t')(t + t') = (a_3 - a_3') s^2 d^6 + (a_4 - a_4') s d^8 + (a_5 - a_5') d^{10},$$ and so $$|t - t'|\ll \frac{T^3 |s|^2 d^6}{t + t'}.$$ Since $t^2 = s^5 + a_2 s^3 d^4 + a_3 s^2 d^6 + a_4 s d^8 + a_5 d^{10}$ and $|s|\geq \delta^{-\delta^{-1}} T d^2$, it follows that, once $\delta\ll 1$ (i.e.\ once $\delta$ is sufficiently small), $t\asymp |s|^{\frac{5}{2}}$. Thus we have found that $$|t - t'|\ll T^3 |s|^{-\frac{1}{2}} d^6.$$ Moreover, since $$t^2\equiv s^5 + a_2 s^3 d^4\pmod{d^6},$$ and there are $\ll_\eps d^\eps$ many such square-roots of $s^5 + a_2 s^3 d^4$ modulo $d^6$ (there are $\ll 2^{\omega(d)}$, to be precise, since $(s,d) = 1$), it follows that any such $t$ falls into $\ll_\eps d^\eps\ll_\eps T^\eps$ congruence classes modulo $d^6$ and inside one of two (depending on sign) intervals of length $\ll T^3 |s|^{-\frac{1}{2}} d^6$, whence the number of $t$ extending $(s,d,a_2)$ is $$\ll_\eps T^\eps \left(1 + \frac{T^3}{|s|^{\frac{1}{2}}}\right).$$

Finally, since $|a_2|\ll T^2$, of course there are only $\ll T^2$ many $a_2$ extending $(s,d)$.

Hence, in sum, we have shown that (on multiplying these bounds together and summing over the possible $s$ and $d$)
\begin{align*}
\sum_{f\in \FunivT} \#|\I_f^\up|&\ll_\eps T^{2+\eps}\sum_{T\ll |s|\ll T^{c_\up - \delta}} \sum_{d\ll |s|^{\frac{1}{2}} T^{-\frac{1}{2}}} \left(1 + \frac{T^3}{|s|^{\frac{1}{2}}}\right)\left(1 + \frac{T^4}{|s|}\right)\left(1 + \frac{T^5}{|s|}\right)
\\&\ll_\eps T^{2+\eps} \sum_{T\ll |s|\ll T^{c_\up - \delta}} |s|^{\frac{1}{2}} T^{-\frac{1}{2}} + T^{\frac{5}{2}} + \frac{T^{\frac{9}{2}}}{|s|^{\frac{1}{2}}} + \frac{T^{\frac{15}{2}}}{|s|} + \frac{T^{\frac{17}{2}}}{|s|^{\frac{3}{2}}} + \frac{T^{\frac{23}{2}}}{|s|^2}
\\&\ll_\eps T^\eps\left(T^{\frac{3}{2}(c_\up - \delta) + \frac{3}{2}} + T^{(c_\up - \delta) + \frac{9}{2}} + T^{\frac{1}{2}(c_\up - \delta) + \frac{13}{2}} + T^{\frac{25}{2}}\right),
\end{align*}
and this is admissible (remember that the size of the family is $T^{14}$) since $c_\up = \frac{25}{3}$.
\end{proof}

Now we turn to those small points that do \emph{not} have such large $x$-coordinates. The ideas for this case are much the same as the ideas for the previous one.
\begin{lem}\label{the truly little guys}
$$\sum_{f\in \FunivT} \#|\I_f^\down|\ll T^{14 - \Omega(\delta)}.$$
\end{lem}

\begin{proof}
In the same way as Lemma \ref{big lookin small points}, we write
\begingroup
\addtolength{\jot}{-0.2in}
\begin{align*}
\sum_{f\in \FunivT, H(f)\asymp T} \#|\I_f^\down|\leq \#|\{&(s,d,a_2,t,a_3,a_4,a_5) : t^2 = s^5 + a_2 d^4s^3 + \cdots + a_5 d^{10}, \\&|a_i|\leq T^i, \Delta_f\neq 0, (s,d) = 1, (t,d) = 1, |s|\leq T^{c_\down - \delta}, d^2 > \delta^{\delta^{-1}} |s| T^{-1}\}|.
\end{align*}
\endgroup

And, again, we argue by iteratively bounding the number of integers extending a tuple, where our definition of \emph{extending} is much the same as before (except the ambient set has changed). First, the number of $a_5$ extending $(s,d,a_2,t,a_3,a_4)$ is $\leq 1$, again because we can solve for $a_5 d^{10}$ and hence $a_5$ in the defining equation.

Now, if both $a_4$ and $a_4'$ extend $(s,d,a_2,t,a_3)$ to $(s,d,a_2,t,a_3,a_4,a_5)$ and $(s,d,a_2,t,a_3,a_4',a_5')$, then, on taking differences of the defining equations, we find that $$0 = (a_4 - a_4') s d^8 + (a_5 - a_5') d^{10},$$ and so $0\equiv (a_4 - a_4') s d^8\pmod{d^{10}}$ --- i.e., $$a_4\equiv a_4'\pmod{d^2},$$ in much the same way as before. This time we simply use that $|a_4|\ll T^4$ to get that the number of $a_4$ extending $(s,d,a_2,t,a_3)$ is $$\ll 1 + \frac{T^4}{d^2}$$ since if there is at least one solution, any other must lie in an interval of length $T^4$ intersected with a congruence class modulo $d^2$.

In the same way, if $a_3$ and $a_3'$ extend $(s,d,a_2,t)$, we find that $$a_3\equiv a_3'\pmod{d^2},$$ and since $|a_3|\ll T^3$ it follows that the number of $a_3$ extending $(s,d,a_2,t)$ is $$\ll 1 + \frac{T^3}{d^2}.$$

Now if $t$ extends $(s,d,a_2)$, then since $$t^2\equiv s^5 + a_2 s^3 d^4\pmod{d^6}$$ and there are $\ll_\eps d^\eps\ll_\eps T^\eps$ many square roots of an element of $(\Z/d^6)^\times$, we find that $t$ must fall in one of $\ll_\eps T^\eps$ congruence classes modulo $d^6$. Also, since $$t^2\ll T^5 d^{10}$$ (since $d^2\gg T^{-1} |s|$) via the defining equation, it follows that $$|t|\ll T^{\frac{5}{2}} d^5.$$ Hence the number of $t$ extending $(s,d,a_2)$ is $$\ll_\eps T^\eps \left(1 + \frac{T^{\frac{5}{2}}}{d}\right).$$

Finally, $|a_2|\ll T^2$ implies the number of $a_2$ extending $(s,d)$ is of course $\ll T^2$.

In sum, we have found that (here we use that we will choose $\delta\asymp 1$ in the end)
\begin{align*}
\sum_{f\in \FunivT} \#|\I_f^\down|&\ll_\eps T^{2+\eps}\sum_{d\leq T^{\frac{c_\down - \delta}{2}}}\sum_{|s|\ll \min(T^{c_\down - \delta}, \delta^{-\delta^{-1}} T d^2)} \left(1 + \frac{T^{\frac{5}{2}}}{d}\right)\left(1 + \frac{T^3}{d^2}\right)\left(1 + \frac{T^4}{d^2}\right)
\\&\ll_\eps T^{2+\eps} \sum_{T^{\frac{c_\down - 1 - \delta}{2}}\ll d\leq T^{\frac{c_\down - \delta}{2}}}\sum_{|s|\ll T^{c_\down - \delta}} \left(1 + \frac{T^{\frac{5}{2}}}{d}\right)\left(1 + \frac{T^3}{d^2}\right)\left(1 + \frac{T^4}{d^2}\right)
\\&\quad+ T^{2+\eps} \sum_{d\ll T^{\frac{c_\down - 1 - \delta}{2}}}\sum_{|s|\ll T d^2} \left(1 + \frac{T^{\frac{5}{2}}}{d}\right)\left(1 + \frac{T^3}{d^2}\right)\left(1 + \frac{T^4}{d^2}\right)
\\&\ll_\eps T^\eps \left(T^{\frac{3(c_\down - \delta)}{2} + 2} + T^{c_\down + \frac{9}{2} - \delta} + T^{\frac{19}{2}} + T^{\frac{c_\down - 1 - \delta}{2} + 7} + T^{10} + T^{\frac{25}{2}}\right),
\end{align*}
and this is admissible since $c_\down = 8$.
\end{proof}

This completes the small point counting.

\subsection{Medium points.}
We recall the notation $C_f: y^2 = f(x), J_f := \Jac(C_f), K_f := J_f/\{\pm 1\}$ for the curve, Jacobian, and Kummer variety associated to $f$. Write $\infty\in C_f(\Q)$ for the marked rational Weierstrass point on $C_f\subseteq \P^2$. Write $j: C_f\to K_f$ for the Abel-Jacobi map associated to $\infty$, and write $\kappa: C_f\to K_f$ for $j: C_f\to J_f\surj K_f$, i.e.\ $j$ composed with the canonical projection $J_f\to K_f$. We will embed $K_f\subseteq \P^3$ as in Cassels-Flynn \cite{casselsflynn}, which realizes the map $C_f\to \P^3$ as the projection $x: C_f\to \P^1$ composed with a Veronese embedding $\P^1\to \P^2\subseteq \P^3$, with image a rational normal curve in a hyperplane. Explicitly, $C_f\to K_f\to \P^3$ is the map $(x,y)\mapsto [0,1,x,x^2]$. We define, for $P\in C_f(\Qbar)$, $h(P) := h(x(P))$, the logarithmic Weil height of the $x$-coordinate of $P$. We define $h_K$ to be the pullback of the logarithmic Weil height on $\P^3$ to $K_f$ via the Cassels-Flynn embedding, and, for $P\in K_f(\Qbar)$, $\hat{h}(P) := \lim_{n\to\infty} \frac{h_K(2^n P)}{4^n}$, the canonical height on $K_f$ (or $J_f$, on pulling back via the projection). We will omit $j$ and $\kappa$ in expressions like $h_K(P+Q)$ or $h_K(P)$ for $P,Q\in C_f(\Q)$ --- of course these mean $h_K(j(P)+j(Q))$ (with the inner expression projected to $K_f$) and $h_K(\kappa(P))$, respectively, but since there will always be a unique sensible interpretation there will be no ambiguity in dropping the various embeddings.

Note that from the simple fact that, for $P\in C_f(\C)$, $\kappa(P) = [0, 1, x(P), x(P)^2]$, it follows that $h_K(\kappa(P)) = 2 h(P)$ for all $P\in C_f(\Qbar)$. Thus e.g.\ our small point counting allows us to focus only on those $P\in C_f(\Q)$ with $h_K(\kappa(P))\geq (16 - 2\delta) h(f)$, since $c_\up\geq c_\down = 8$.

Having set all this notation, let us turn to bounding the number of medium points on these curves. To do this we will establish an explicit Mumford gap principle via the explicit addition law on $J_f\inj \P^{15}$ provided by Flynn \cite{flynnsite} on his website, and then use the usual sphere packing argument to conclude the section.

\subsubsection{Upper bounds on $\hat{h}(P+Q)$.}

First let us deal with the gap principle. We will first prove upper bounds on expressions of the form $\hat{h}(P+Q)$. We split into two cases, depending on whether or not the given points have unusually (Archimedeanly) large $x$-coordinate or not.
\begin{lem}\label{the upper bound for sums of normal points}
Let $P\neq \pm Q\in C_f(\Q)$ with $|x(P)|, |x(Q)|\ll \delta^{-\delta^{-1}} H(f)$ and $h(P)\geq h(Q)$. Then: $$\hat{h}(P+Q)\leq 3h_K(P) + (5 + O(\delta)) h(f).$$
\end{lem}

\begin{proof}
Write $P =: (X,Y) =: \left(\frac{S}{D^2}, \frac{U}{D^5}\right), Q =: (x,y) =: \left(\frac{s}{d^2}, \frac{u}{d^5}\right)$, with $(S,D) = (s,d) = 1$. By Flynn's explicit formulas, we find that: $$\kappa(P+Q) = [(X-x)^2, (X-x)^2 (X+x), (X-x)^2 X x, 2a_5 + a_4(X+x) + 2a_3 Xx + a_2 Xx (X+x) + X^2 x^2 (X+x) - 2Yy].$$ Clearing denominators, we find that:
\begin{align*}
\kappa(P+Q) = [&D^2d^2 (Sd^2-sD^2)^2, (Sd^2 - sD^2)^2 (Sd^2 + sD^2), Ss (Sd^2 - sD^2)^2, \\&2a_5 D^6d^6 + a_4 D^4d^4 (Sd^2 + sD^2), 2a_3 D^4d^4 Ss + a_2 D^2d^2 Ss (Sd^2 + sD^2) + S^2s^2 (Sd^2 + sD^2) - 2DdUu].
\end{align*}

It therefore follows that, with these affine coordinates,
\begin{align*}
|\kappa(P+Q)_1|&\ll H_K(P)^3,\\
|\kappa(P+Q)_2|&\ll H_K(P)^3,\\
|\kappa(P+Q)_3|&\ll H_K(P)^3,\\
|\kappa(P+Q)_4|&\ll H(f)^5 H_K(P)^3,
\end{align*}
since $|S|,|s|,D^2,d^2\ll H_K(P)^{\frac{1}{2}}$ and $|X|, |x|\ll H(f)$ by hypothesis.

Next, for $R = [k_1, k_2, k_3, k_4]\in K_f$, we have that $2R = [\delta_1(R), \delta_2(R), \delta_3(R), \delta_4(R)]$, where:
\begin{align*}
\delta_1(k_1,\ldots,k_4) &:= 
4 a_2^2 a_5 k_1^4 + 8 a_4^2 k_1^3 k_2 - 32 a_3 a_5 k_1^3 k_2 - 
 8 a_2 a_5 k_1^2 k_2^2 + 4 a_5 k_2^4 - 16 a_2 a_5 k_1^3 k_3 \\&- 
 4 a_2 a_4 k_1^2 k_2 k_3 - 16 a_5 k_1 k_2^2 k_3 + 4 a_4 k_2^3 k_3 + 
 16 a_5 k_1^2 k_3^2 - 8 a_4 k_1 k_2 k_3^2 + 4 a_3 k_2^2 k_3^2 + 4 a_2 k_2 k_3^3 \\&+ 
 4 a_2 a_4 k_1^3 k_4 - 32 a_5 k_1^2 k_2 k_4 - 4 a_4 k_1 k_2^2 k_4 - 
 8 a_4 k_1^2 k_3 k_4 - 8 a_3 k_1 k_2 k_3 k_4 - 4 a_2 k_1 k_3^2 k_4 + 8 k_3^3 k_4 \\&+ 
 4 a_3 k_1^2 k_4^2 - 4 k_2 k_3 k_4^2 + 4 k_1 k_4^3,\\ 
\delta_2(k_1,\ldots,k_4) &:= 
a_2 a_4^2 k_1^4 - 4 a_2 a_3 a_5 k_1^4 + 16 a_5^2 k_1^4 - 4 a_2^2 a_5 k_1^3 k_2 \\&+ 
 16 a_4 a_5 k_1^3 k_2 + 4 a_4^2 k_1^2 k_2^2 - 4 a_2 a_5 k_1 k_2^3 - 
 6 a_2^2 a_4 k_1^3 k_3 + 16 a_4^2 k_1^3 k_3 - 32 a_3 a_5 k_1^3 k_3 \\&+ 
 16 a_3 a_4 k_1^2 k_2 k_3 - 20 a_2 a_5 k_1^2 k_2 k_3 - 8 a_2 a_4 k_1 k_2^2 k_3 + 
 8 a_5 k_2^3 k_3 + 5 a_2^3 k_1^2 k_3^2 + 16 a_3^2 k_1^2 k_3^2 \\&- 
 14 a_2 a_4 k_1^2 k_3^2 - 12 a_2 a_3 k_1 k_2 k_3^2 + 32 a_5 k_1 k_2 k_3^2 + 
 4 a_4 k_2^2 k_3^2 - 6 a_2^2 k_1 k_3^3 + 16 a_4 k_1 k_3^3 + a_2 k_3^4 \\&+ 
 4 a_2 a_5 k_1^3 k_4 + 2 a_2 a_4 k_1^2 k_2 k_4 + 8 a_5 k_1 k_2^2 k_4 + 
 4 a_4 k_2^3 k_4 - 12 a_2 a_3 k_1^2 k_3 k_4 - 16 a_5 k_1^2 k_3 k_4 \\&- 
 10 a_2^2 k_1 k_2 k_3 k_4 + 16 a_4 k_1 k_2 k_3 k_4 + 8 a_3 k_2^2 k_3 k_4 + 
 16 a_3 k_1 k_3^2 k_4 + 2 a_2 k_2 k_3^2 k_4 + 4 a_4 k_1^2 k_4^2 \\&+ 
 8 a_3 k_1 k_2 k_4^2 + 5 a_2 k_2^2 k_4^2 - 8 a_2 k_1 k_3 k_4^2 + 4 k_3^2 k_4^2 + 
 4 k_2 k_4^3,\\ 
\delta_3(k_1,\ldots,k_4) &:= 
4 a_3 a_4^2 k_1^4 - 16 a_3^2 a_5 k_1^4 + 8 a_2 a_4 a_5 k_1^4 + 
 4 a_2 a_4^2 k_1^3 k_2 - 16 a_2 a_3 a_5 k_1^3 k_2 - 32 a_5^2 k_1^3 k_2 \\&- 
 8 a_4 a_5 k_1^2 k_2^2 + 4 a_4^2 k_1 k_2^3 - 16 a_3 a_5 k_1 k_2^3 - 
 8 a_2^2 a_5 k_1^3 k_3 - 16 a_4 a_5 k_1^3 k_3 - 8 a_4^2 k_1^2 k_2 k_3 \\&- 
 24 a_2 a_5 k_1 k_2^2 k_3 - 8 a_3 a_4 k_1^2 k_3^2 + 24 a_2 a_5 k_1^2 k_3^2 - 
 4 a_2 a_4 k_1 k_2 k_3^2 + 12 a_5 k_2^2 k_3^2 - 16 a_5 k_1 k_3^3 \\&+ 
 8 a_4 k_2 k_3^3 + 4 a_3 k_3^4 + 8 a_4^2 k_1^3 k_4 - 32 a_3 a_5 k_1^3 k_4 - 
 24 a_2 a_5 k_1^2 k_2 k_4 - 8 a_5 k_2^3 k_4 - 4 a_2 a_4 k_1^2 k_3 k_4 \\&- 
 24 a_5 k_1 k_2 k_3 k_4 - 4 a_4 k_2^2 k_3 k_4 - 8 a_4 k_1 k_3^2 k_4 + 
 4 a_2 k_3^3 k_4 - 12 a_5 k_1^2 k_4^2 - 4 a_4 k_1 k_2 k_4^2 + 4 k_3 k_4^3,\\ 
\delta_4(k_1,\ldots,k_4) &:= 
a_2^2 a_4^2 k_1^4 - 2 a_4^3 k_1^4 - 4 a_2^2 a_3 a_5 k_1^4 + 8 a_3 a_4 a_5 k_1^4 \\&- 
 16 a_2 a_5^2 k_1^4 - 8 a_3 a_4^2 k_1^3 k_2 + 32 a_3^2 a_5 k_1^3 k_2 - 
 16 a_2 a_4 a_5 k_1^3 k_2 - 2 a_2 a_4^2 k_1^2 k_2^2 + 8 a_2 a_3 a_5 k_1^2 k_2^2 \\&+ 
 16 a_5^2 k_1^2 k_2^2 + a_4^2 k_2^4 - 4 a_3 a_5 k_2^4 - 4 a_2 a_4^2 k_1^3 k_3 + 
 16 a_2 a_3 a_5 k_1^3 k_3 + 32 a_5^2 k_1^3 k_3 + 8 a_2^2 a_5 k_1^2 k_2 k_3 \\&+ 
 16 a_4 a_5 k_1^2 k_2 k_3 - 8 a_2 a_5 k_2^3 k_3 + 12 a_4^2 k_1^2 k_3^2 - 
 32 a_3 a_5 k_1^2 k_3^2 - 8 a_2 a_5 k_1 k_2 k_3^2 - 2 a_2 a_4 k_2^2 k_3^2 \\&- 
 4 a_2 a_4 k_1 k_3^3 - 8 a_5 k_2 k_3^3 + a_2^2 k_3^4 - 2 a_4 k_3^4 - 
 8 a_2^2 a_5 k_1^3 k_4 - 12 a_4^2 k_1^2 k_2 k_4 + 48 a_3 a_5 k_1^2 k_2 k_4 \\&+ 
 8 a_2 a_5 k_1 k_2^2 k_4 + 16 a_2 a_5 k_1^2 k_3 k_4 + 4 a_2 a_4 k_1 k_2 k_3 k_4 - 
 16 a_5 k_2^2 k_3 k_4 - 8 a_5 k_1 k_3^2 k_4 - 12 a_4 k_2 k_3^2 k_4 \\&- 
 8 a_3 k_3^3 k_4 - 2 a_2 a_4 k_1^2 k_4^2 + 8 a_5 k_1 k_2 k_4^2 - 
 2 a_2 k_3^2 k_4^2 + k_4^4.
\end{align*}

In particular, $(2R)_i = \delta_i(R)$ is a homogenous polynomial (with $\Z$ coefficients) of degree $(12+i, 4)$, where we give the $a_i$ degree $(i,0)$ and the $k_i$ degree $(i,1)$. Hence $(2^N R)_i$ is homogeneous of degree $(4^{N+1} - 4 + i, 4^N)$ under this grading, with integral coefficients (thought of as a polynomial in $\Z[a_2,\ldots,a_5,k_1,\ldots,k_4]$) of size $O_N(1)$.

It therefore follows that:
\begin{align*}
|\kappa(2^N (P+Q))_i|&\ll_N \max_{\alpha_1 + \alpha_2 + \alpha_3 + \alpha_4 = 4^N} H(f)^{4^{N+1} - 4 + i - \alpha_1 - 2\alpha_2 - 3\alpha_3 - 4\alpha_4} H_K(P)^{3\cdot 4^N} H(f)^{5\alpha_4}
\\&\ll H(f)^{4^{N+1} + 4^N} H_K(P)^{3\cdot 4^N},
\end{align*}
which is to say that (since $H_K(R)\leq \max_i |R_i|$ when all the coordinates $R_i\in \Z$) $$\frac{1}{4^N} h_K(2^N(P+Q))\leq 3h_K(P) + 5h(f) + O_N(1).$$

Now since we'd like a result about the canonical height, observe that:
\begin{align*}
\hat{h}(P+Q) - \frac{1}{4^N} h_K(2^N(P+Q)) &= \sum_{n\geq N} 4^{-n-1}\left(h_K(2^{n+1}(P+Q)) - 4 h_K(2^n (P+Q))\right)
\\&= \sum_{n\geq N} 4^{-n-1}\left(\lambda_\infty(2^{n+1}(P+Q)) - 4\lambda_\infty(2^n(P+Q))\right) \\&\quad+ \sum_p\sum_{n\geq N} 4^{-n-1}\left(\lambda_p(2^{n+1}(P+Q)) - 4\lambda_p(2^n(P+Q))\right),
\end{align*}
where $\lambda_v(R) := \max_i \log{|R_i|_v}$ are the local na{\" i}ve heights.

Now, by Stoll \cite{stollii} (see Corollary \ref{stoll's bound}), we have that $$|\lambda_p(2R) - 4\lambda_p(R)|\ll v_p(\Delta_f)\log{p}$$ (and the difference is $0$ if $v_p(\Delta_f)\leq 1$), whence this expression simplifies to
\begin{align*}
\hat{h}(P+Q) - \frac{1}{4^N} h_K(2^N(P+Q)) = \sum_{n\geq N} 4^{-n-1}\left(\lambda_\infty(2^{n+1}(P+Q)) - 4\lambda_\infty(2^n(P+Q))\right) + O(4^{-N}\log{|\Delta_f|}).
\end{align*}

Moreover, certainly (by examining the explicit expressions for $2R$ for $R\in K_f(\Q)$) $$\lambda_\infty(2R)\leq 4\lambda_\infty(R) + O(h(f)).$$ Thus we find that the full expression is $$\ll 4^{-N} h(f),$$ since $$|\Delta_f|\ll H(f)^{20}.$$

Hence we have that
\begin{align*}
\hat{h}(P+Q)&\leq \frac{1}{4^N} h_K(2^{N}(P+Q)) + O(4^{-N} h(f))
\\&\leq 3 h_K(P) + (5 + O(4^{-N})) h(f) + O_N(1).
\end{align*}
Taking $N\asymp \log{(\delta^{-1})}$ (which is $\asymp 1$) gives the result.
\end{proof}

Now let us handle the case of points with large $x$-coordinate.
\begin{lem}\label{the upper bound for sums of big points}
Let $P\neq \pm Q\in C_f(\Q)$ with $|x(P)|, |x(Q)|\gg \delta^{-\delta^{-1}} H(f)$ and $h(P)\geq h(Q)$. Then: $$\hat{h}(P+Q)\leq 3h_K(P) - h(f) + O(\delta h(f)).$$
\end{lem}

\begin{proof}
The proof follows in the same way as the previous Lemma, except now from $$\kappa(P+Q) = [(X-x)^2, (X-x)^2 (X+x), (X-x)^2 X x, 2a_5 + a_4(X+x) + 2a_3 Xx + a_2 Xx (X+x) + X^2 x^2 (X+x) - 2Yy]$$ we see that $$|\kappa(P+Q)_i|\ll \max(|X|, |x|)^{i+1},$$ without a factor of $H(f)^5$ for $i=4$ (as we had last time), since both $|x|$ and $|X|$ are so large. Thus, in the same way as the previous Lemma, we find that $$|\kappa(2^N(P+Q))_i|\ll_N \max_{\alpha_1 + \cdots + \alpha_4 = 4^N} H(f)^{4^{N+1} - 4 + i - \alpha_1 - 2\alpha_2 - 3\alpha_3 - 4\alpha_4}\cdot \max(|X|, |x|)^{2\alpha_1 + 3\alpha_2 + 4\alpha_3 + 5\alpha_4 + 4^N}.$$ Now since $\max(|X|, |x|)\gg H(f)$, we find that therefore $$|\kappa(2^N(P+Q))_i|\ll_N \max(|X|, |x|)^{5\cdot 4^N}\cdot H(f)^{-4+i}.$$

Now since $D^6 d^6 \kappa(P+Q)_i\in \Z$, we see that $D^{6\cdot 4^N} d^{6\cdot 4^N} \kappa(2^N(P+Q))_i\in \Z$, and so $$\frac{1}{4^N} h_K(2^N(P+Q))\leq 6\log{D} + 6\log{d} + 5\log{\max(|X|, |x|)} + O_N(1) + O(4^{-N} h(f)).$$ Since $h(P) = 2\log{D} + \log{|X|}$, we see that $h_K(P) = 4\log{D} + 2\log{|X|}$, and similarly for $h_K(Q)$. Thus we find that $$\frac{1}{4^N} h_K(2^N(P+Q))\leq 3 h_K(P) - \log{\max(|X|, |x|)} + O_N(1) + O(4^{-N} h(f))\leq 3 h_K(P) - h(f) + O_N(1) + O(4^{-N} h(f)).$$ The rest of the argument (to turn this into a bound on the canonical height) is the same.
\end{proof}

This completes the upper bounding of $\hat{h}(P+Q)$ required for the gap principle. We will also need lower bounds on $\hat{h}(P-Q)$, which will require a detailed study of what we will call $\hat{\lambda}_\infty$ --- we were able to get away with not dealing with it because it is much easier to upper bound the sizes of $\delta_i(k_1,\ldots,k_4)$ than to lower bound them.

\subsubsection{Definition of local 'heights' and Stoll's bound.}

To begin with, let us define the local canonical 'height' functions.

\begin{defn}
Let $$\hat{\lambda}_v(\bullet) := \lambda_v(\bullet) + \sum_{n\geq 0} 4^{-n-1} \left(\lambda_v(2^{n+1}\bullet) - \lambda_v(2^n\bullet)\right),$$ the \emph{local canonical ``height''} at a place $v$.
\end{defn}
(Here $\lambda_v(\bullet) := \max(\log{|\bullet|_v}, 0)$.) Note that this sum converges by the inequality given in (7.1) (i.e., the treatment of the Archimedean case --- at finite primes Stoll's explicit bounds on the local height difference have already been mentioned) in \cite{stolli}. Indeed, evidently one has the upper bound $\lambda_v(2P) - \lambda_v(P)\leq O(h(f))$, and the corresponding lower bound (and thus a two-sided bound uniform in $P$) is given by (7.1). For completeness, note that the uniformity in $P$ follows from e.g.\ Formulas 10.2 and 10.3 in \cite{stolli} --- the bound only depends on the roots of $f$. It follows therefore that each term in the sum is bounded in absolute value by a constant depending only on $f$ (in fact, the constant is $\ll h(f)$ as well), whence convergence.

We have written the word height in quotes because these $\hat{\lambda}_v$ are not functions on $K_f\subseteq \P^3$, but rather on a lift (via the canonical projection $\A^4 - \{0\}\surj \P^3$) that we might (and will) call $\tilde{K}_f$, the cone on $K_f$ in $\A^4$ --- i.e.\ the subvariety of $\A^4$ defined by the same defining quartic. That is, these functions $\hat{\lambda}_v$ \emph{do} change under scaling homogeneous coordinates, but they at least do so in a controlled fashion --- indeed, so that $\sum_v \hat{\lambda}_v = \hat{h}$ remains invariant.

Now we may state Stoll's \cite{stollii} bound on the local height differences at finite primes.
\begin{thm}(Stoll, \cite{stollii})\label{stoll's theorem}
$$|\hat{\lambda}_p - \lambda_p|\leq \begin{cases} \frac{1}{3} v_p(2^4 \Delta_f) \log{p} & v_p(\Delta_f)\geq 2\\ 0 & v_p(\Delta_f)\leq 1.\end{cases}$$
\end{thm}

Note that this will be all we use to handle the local heights at finite places. Note also that it immediately follows that:
\begin{cor}\label{stoll's bound}
$$\sum_p |\hat{\lambda}_p - \lambda_p|\leq \frac{1}{3}\log{|\Delta_f|} + O(1).$$
\end{cor}

\subsubsection{Analysis of $\hat{\lambda}_\infty$ and the partition at $\infty$.}

Thus we are left with studying $\hat{\lambda}_\infty$. Following Pazuki \cite{pazuki}, we will relate $\hat{\lambda}_\infty$ to a Riemann theta function plus the logarithm of a certain linear form (which is implicit in his normalization of Kummer coordinates), and then prove that the theta function term is harmless and may be ignored. Thus $\hat{\lambda}_\infty$ will be related to a linear form involving roots of our original quintic polynomial $f$, except that we will be able to choose which roots we consider. (This corresponds to translating our original point by a suitable two-torsion point.) We will then prove that there is at least one choice for which the resulting height is as desired, completing the argument.

So let $$\tau_f =: \mat{\tau_1}{\tau_{12}}{\tau_{12}}{\tau_2}\in \Sym^2(\C^2)$$ be the Riemann matrix corresponding to $J_f$ in the Siegel fundamental domain $\mathcal{F}_2$ --- that is, so that $\tau_f$ is a symmetric $2\times 2$ complex matrix whose imaginary part is positive-definite, and so that $||\Re{\tau}||\ll 1$, $$\Im{\tau_2}\geq \Im{\tau_1}\geq 2\Im{\tau_{12}} > 0,$$ and $$\Im{\tau_1}\geq \frac{\sqrt{3}}{2}.$$

Let $$\Psi_f: \C^2/(\Z^2 + \tau_f \Z^2)\simeq J_f(\C)$$ be a complex uniformization. Let $(\vec{a}, \vec{b})$ be a theta characteristic (i.e., simply $\vec{a}, \vec{b}\in \frac{1}{2}\Z^2/\Z^2$). Let $\theta_{\vec{a}, \vec{b}}: \C^2\to \C^2$ via: $$\theta_{\vec{a}, \vec{b}}(Z) := \sum_{\vec{n}\in \Z^2} e\left(\frac{1}{2}\pair{\vec{n} + \vec{a}}{\tau_f\cdot (\vec{n} + \vec{a})} + \pair{\vec{n} + \vec{a}}{Z + \vec{b}}\right),$$ where $e(z) := e^{2\pi i z}$. This is the \emph{Riemann theta function} associated to the theta characteristic $(\vec{a}, \vec{b})$.

Notice that $\theta_{\vec{a}, \vec{b}}$ is an even function if and only if $(2\vec{a})\cdot (2\vec{b})\equiv 0\pmod{2}$ --- indeed, one easily sees (via the change of variable $\vec{n}\mapsto -\vec{n} - 2\vec{a}$) that $\theta_{\vec{a},\vec{b}}(-Z) = (-1)^{4\pair{\vec{a}}{\vec{b}}} \theta_{\vec{a},\vec{b}}(Z)$. Characteristics $(\vec{a},\vec{b})$ for which $(2\vec{a})\cdot (2\vec{b})\equiv 0\pmod{2}$ are called \emph{even}, and those for which this fails are called \emph{odd}. Evidently there are exactly ten even theta characteristics, namely $\vec{a} = \vec{b} = 0$, the six more with either $\vec{a} = 0$ or $\vec{b} = 0$, the two with $\{\vec{a}, \vec{b}\} = \{(\frac{1}{2},0), (0,\frac{1}{2})\}$, and $\vec{a} = \vec{b} = (\frac{1}{2},\frac{1}{2})$. The remaining six characteristics are odd.

Given a root $\alpha$ of $f$, we will write $R_\alpha := (\alpha, 0)\in C_f(\Qbar)$ for the corresponding point on $C_f$, whose image in $J_f$ is two-torsion. If $\alpha = \infty$ we will interpret $R_\alpha = \infty$. Note that $$J_f(\C)[2] = \{0\}\cup \{[R_\alpha] - [\infty] : f(\alpha) = 0\}\cup \{[R_\alpha] + [R_\beta] - 2[\infty] : \alpha\neq \beta, f(\alpha) = f(\beta) = 0\}.$$ For ease of notation we will write $$Q_{\alpha,\beta} := [R_\alpha] + [R_\beta] - 2[\infty]\in J_f(\Qbar).$$

Let now $$\Theta := \im(j) = \im(C_f\to J_f) = \{[P] - [\infty] : P\in C_f(\C)\}\subseteq J_f(\C),$$ the \emph{theta divisor} of $C_f$ in $J_f$. Regarding this theta divisor and these $R_\alpha$ we have the following famous theorem of Riemann \cite{riemann}:
\begin{thm}[Riemann]
Let $(\vec{a}, \vec{b})$ be a theta characteristic. Then: there exists a unique $Q\in J_f(\C)[2]$ such that $\div_0(\theta_{\vec{a},\vec{b}}) = \Theta + Q$, where $\div_0$ is the divisor of zeroes. Moreover, the odd theta characteristics are exactly those that correspond to a point in the image of $C_f$ --- i.e., either $0$ or one of the form $[R_\alpha] - [\infty]$.\footnote{Indeed these are the only two-torsion points in the image of $C_f$, since otherwise the existence of a nontrivial linear equivalence $R_\alpha + R_\beta\sim P + \infty$ would imply that there is a nonconstant meromorphic function $f$ on $C_f(\C)$ with $\div_\infty(f)\leq P+\infty$, which is contradicted by Riemann-Roch ($h^0(\infty-P) = 0$ since otherwise $C_f$ would have a degree $1$ map to $\P^1$, thus have genus $0$).}
\end{thm}
Here we have used the abuse of notation $\Theta + Q := \{([P] - [\infty]) + Q: P\in C_f(\C)\}$ for the translate of the theta divisor by the two-torsion point $Q$. (Note in particular that, for an even theta characteristic $(\vec{a},\vec{b})$, $\theta_{\vec{a},\vec{b}}(0)\neq 0$ since $0\not\in \Theta + Q_{\alpha,\beta}$! Indeed, otherwise we would have some $P$ for which $h^0(P+\infty) = 2$, which would be a contradiction by Riemann-Roch.)

Of course by considering cardinalities we see that the six odd characteristics are in bijection with the five roots of $f$ plus the point at infinity, and the ten even characteristics are likewise in bijection with the ten pairs of finite roots of $f$.

Note also that a natural question is which characteristic $(\vec{a}, \vec{b})$ has $\div_0(\theta_{\vec{a},\vec{b}}) = \Theta$ --- i.e., which characteristic corresponds to $0\in J_f(\C)$. The answer is given by a theorem of Mumford (see e.g.\ Theorem 5.3 in \cite{mumfordtataii}).\footnote{This is not strictly necessary for our arguments, though it is convenient to use the form of $\chi_\infty$ below.}
\begin{thm}[Mumford]
$$\div_0(\theta_{\left(\frac{1}{2}, \frac{1}{2}\right), \left(0, \frac{1}{2}\right)}) = \Theta.$$
\end{thm}

Henceforth we will write $$\chi_\infty := \left(\left(\frac{1}{2}, \frac{1}{2}\right), \left(0, \frac{1}{2}\right)\right),$$ call our characteristics $\chi$, and, for $\chi = (\vec{a}, \vec{b})$, we will write $\chi_a := \vec{a}$ and $\chi_b := \vec{b}$. We will also write $\chi_\rho$ for the odd characteristic corresponding to the root $\rho$ of $f$. We will further write $P_\rho := j(R_\rho)$, and $\tilde{P}_\rho := (0, 1, \rho, \rho^2)\in \tilde{K}_f(\C)$ for a lift of $P_\rho$, regarded as a point of $K_f(\C)$, to $\A^4$. Let also $\ell_\rho$ be the following linear form: $$\ell_\rho(w,x,y,z) := \rho^2 w - \rho x + y.$$

Let us identify a lift of $P_\rho$ along our uniformization $\Psi_f: \C^2/(\Z^2 + \tau_f\cdot \Z^2)\simeq J_f(\C)$. Write, for $\chi$ a theta characteristic, $$\tilde{\chi} := \chi_b + \tau_f\cdot \chi_a\in \C^2.$$
\begin{lem}
$$P_\rho = \Psi_f(\tilde{\chi}_\infty + \tilde{\chi}_\rho).$$
\end{lem}
\begin{proof}
Since $K_f(\C)[2]\cap \div_0(\theta_{\chi_\rho})\cap \div_0(\theta_{\chi_\infty}) = \{0, P_\rho\}$, it suffices to show that both these theta functions vanish at this point, since were $\tilde{\chi}_\infty + \tilde{\chi}_\rho\in \Z^2 + \tau_f\cdot \Z^2$, it would follow first that $(\chi_\infty)_a = (\chi_\rho)_a$, and then that $(\chi_\infty)_b = (\chi_\rho)_b$, a contradiction. But, by explicit calculation, for $\chi, \eta$ theta characteristics, $$\theta_\chi(\tilde{\eta}) = \theta_{\chi + \eta}(0)\cdot e\left(-\frac{1}{2}\pair{\eta_a}{\tau_f\cdot \eta_a} - \pair{\eta_a}{\chi_b + \eta_b}\right).$$ Taking $\eta = \chi_\infty + \chi_\rho$ and $\chi = \chi_\rho$ or $\chi_\infty$ gives us the vanishing of both theta functions, as desired.
\end{proof}

We note here that, as is e.g.\ easily verified case-by-case, for $\chi, \chi', \chi''$ distinct odd characteristics, $\chi + \chi' + \chi''$ is even.

Next let us analyze the vanishing locus of $\ell_\rho$.

\begin{lem}
$$\div_0(\ell_\rho)\cap K_f(\C) = 2\div_0(\theta_{\chi_\rho}) = 2(\Theta + P_\rho)\pmod{\pm 1}.$$
\end{lem}

\begin{proof}
Notice that, from the addition law, if $P\in K_f(\C)$ is not $\infty$ or $P_\rho$, then $$\ell_\rho(\kappa(P + P_\rho)_1, \ldots, \kappa(P + P_\rho)_4) = 0.$$ The statement is also true for $P = \infty$ and $P = P_\rho$ by construction. This determines $\ell_\rho$ up to constants (and thus determines its zero divisor uniquely, since e.g.\ $\ell_\rho(P_\gamma)\neq 0$ for $\gamma\neq \rho$ a root of $f$) inside $\kappa^*\O_{\P^3}(1)(K_f)$. Now since the embedding is via $\mathcal{L}_\Theta^{\otimes 2}$, and since $\theta_{\chi_\rho}^2$, a section of this bundle, also satisfies these criteria (up to scaling by a nonzero constant), the conclusion follows from Riemann's theorem. Alternatively, one could see this by explicit computation.
\end{proof}

Now we have enough information to study $\hat{\lambda}_\infty$. For notational ease we will write $$\Xi_\chi(Z):= \theta_\chi(Z)\cdot e^{-\pi\pair{\Im{Z}}{(\Im{\tau_f})^{-1}\cdot \Im{Z}}}.$$ Note that $|\Xi_\chi|$ is invariant under translation by $\Z^2 + \tau_f\cdot \Z^2$, and thus descends to a function on $J_f(\C) - \Supp(\Theta + P_\rho)$, and even further to $K_f(\C) - \Supp(\Theta + P_\rho)$ as well, thanks to the absolute value.

\begin{lem}\label{the formula for the canonical local height}
Let $\rho$ be a root of $f$. Let $\ell_\rho(w,x,y,z) := \rho^2 w - \rho x + y$.\footnote{For the same theorem for other theta characteristics $\chi$, one has to modify the linear form $\ell_\chi$ --- for $\chi_\infty$, $$\ell_\infty(w,x,y,z) := w,$$ and, for $\chi_{\alpha,\beta}$, $$\ell_{\alpha,\beta}(w,x,y,z) := \frac{2a_5 + a_4(\alpha + \beta) + 2a_3\alpha\beta + a_2\alpha\beta(\alpha+\beta) + \alpha^2\beta^2(\alpha+\beta)}{(\alpha-\beta)^2}\cdot w + \alpha\beta\cdot x - (\alpha + \beta)\cdot y + z.$$ The modification of the rest of the theorem (replacing $P_\rho$ with $P_\alpha + P_\beta$, etc.) is straightforward.} Then there exists a constant $c_\rho\in \R$ such that the following formula holds. Let $R\in K_f(\Q) - \Supp(\Theta + P_\rho)$. Let $\tilde{R}\in \tilde{K}_f\subseteq \A^4$ be a lift of $R$ under the canonical projection $\A^4 - \{0\}\surj \P^3$. Let $Z\in [-\frac{1}{2}, \frac{1}{2}]^{\times 2} + \tau_f\cdot [-\frac{1}{2}, \frac{1}{2}]^{\times 2}\subseteq \C^2$ (a fundamental domain of $\Z^2 + \tau_f\cdot \Z^2$) be such that $\Psi_f(Z)\pmod{\pm 1} = R$ under the map $\C^2 / (\Z^2 + \tau_f\cdot \Z^2)\simeq J_f(\C)\surj K_f(\C)$. Then: $$\hat{\lambda}_\infty(\tilde{R}) = -\log{\left|\Xi_{\chi_\rho}(Z)^2\right|} + \log{|\ell_{\rho}(\tilde{R})|} + c_\rho.$$
\end{lem}

\begin{proof}
Note that $\hat{\lambda}_\infty(\tilde{R}) - \log{|\ell_\rho(\tilde{R})|}$ is invariant under scaling, whence it descends to a function on $K_f(\C) - \Supp(\Theta + P_\rho)$. Now note that $$\left[\hat{\lambda}_\infty(2\tilde{R}) - \log{|\ell_\rho(2\tilde{R})|}\right] - 4\left[\hat{\lambda}_\infty(\tilde{R}) - \log{|\ell_\rho(\tilde{R})|}\right] = -\log{\left|\frac{\ell_\rho(2\tilde{R})}{\ell_\rho(\tilde{R})^4}\right|}$$ by definition of $\hat{\lambda}_\infty$.

Now since $$\left[-\log{\left|\Xi_{\chi_\rho}(2Z)^2\right|}\right] - 4\left[-\log{\left|\Xi_{\chi_\rho}(Z)^2\right|}\right] = -\log{\left|\frac{\theta_{\chi_\rho}(2Z)^2}{\theta_{\chi_\rho}(Z)^8}\right|},$$ and in both cases the right-hand sides are of the form $-\log{|G|}$ with $G$ a function on $K_f$ with $\div_0(G) = [2\cdot]^*(2(\Theta + P_\rho)) - 4\cdot (2(\Theta + P_\rho))$, and both $\hat{\lambda}_\infty(\tilde{R}) - \log{|\ell_\rho(\tilde{R})|}$ and $-\log{\left|\Xi_{\chi_\rho}(Z)^2\right|}$ are defined on $K_f(\C) - \Supp(\Theta + P_\rho)$, it follows that the two functions must differ by a constant.\footnote{For a more explicit way to see this, see the attached Mathematica document. The point is that, via Yoshitomi's \cite{yoshitomi} formulas (stated in Grant \cite{grant} and originally from H.F.\ Baker's 1907 book, \cite{hfbaker}), one can express the quotient of theta functions considered above in terms of the $x$- and $y$-coordinates of the corresponding points in the corresponding divisor in the Jacobian, and now one is comparing two rational functions of $x$- and $y$-coordinates. The Mathematica document does this in the case of $\chi = \chi_\infty$ --- the other cases are obtained by translating by the corresponding two-torsion point.} Indeed, since e.g.\ by the Baire category theorem $K_f(\C) - \bigcup_{k\in \Z} [2^k\cdot]^{-1}(\Supp(\Theta + P_\rho))$ is topologically dense (in the Archimedean topology), for $P$ any element of this set we have seen that, writing $F(P) := \hat{\lambda}_\infty(\tilde{P}) - \log{|\ell_\rho(\tilde{P})|} + \log{\left|\Xi_{\chi_\rho}(\Psi_f^{-1}(P))^2\right|}$, $$F(2^{n+1}P) - 4 F(2^n P)$$ is independent of $P$ and $n$, since we have seen that this difference must be constant (since two meromorphic functions with the same divisor of zeroes must differ by a multiplicative constant).

Hence $F(P) = \frac{1}{4^n} F(2^n P) + O_F(4^{-n})$ for all $P$ in this set. By choosing a sequence $n_i$ such that $2^{n_i} P$ converges to a point in $K_f(\C) - \bigcup_{k\in \Z} [2^k\cdot]^{-1}(\Supp(\Theta + P_\rho))$ in the Archimedean topology (one exists since otherwise $P$ must be a torsion point, else its multiples would be dense in a nontrivial abelian subvariety of $K_f(\C)$, whence in particular there would be a subsequence converging to some $2^N P$ with $N$ sufficiently large. But all two-power torsion points are excluded, and any other torsion points will become, after multiplying by a suitably high power of $2$, odd order. But an odd order torsion point has periodic orbit under the multiplication by $2$ map.), we see that $F(P) = 0$, and so, by continuity, we see that in fact $F = 0$ on the whole of $K_f(\C) - \Supp(\Theta + P_\rho)$, as desired.
\end{proof}

Let us apply this to the case of $\chi_\rho$. We find then that:
\begin{cor}\label{first lower bound on c beta}
Let $\alpha\neq \beta$ be roots of $f$. Then: $$c_\beta = 2\log{|\theta_{\chi_\beta + \chi_\infty + \chi_\alpha}(0)|} + \log{\frac{|f'(\alpha)|^{\frac{1}{2}}}{|\alpha - \beta|}}.$$
\end{cor}
\begin{proof}
Apply Lemma \ref{the formula for the canonical local height} to the point $P_\alpha$ and note that $$\hat{\lambda}_\infty(\tilde{P}_\alpha) = \frac{1}{4}\hat{\lambda}_\infty(2\tilde{P}_\alpha) = \frac{1}{2}\log{|f'(\alpha)|},$$ since $$(\delta_1(\tilde{P}_\alpha), \ldots, \delta_4(\tilde{P}_\alpha)) = (0,0,0,f'(\alpha)^2).$$ Finally, we use the already-used fact that, for $\chi, \eta$ theta characteristics, $$\theta_\chi(\tilde{\eta}) = \theta_{\chi + \eta}(0)\cdot e\left(-\frac{1}{2}\pair{\eta_a}{\tau_f\cdot \eta_a} - \pair{\eta_a}{\chi_b + \eta_b}\right).$$
\end{proof}

Let us note here that the (extremely nonobvious) constancy of the right-hand side in $\alpha$ amounts essentially to Thomae's formula for the theta constants of this curve(!).

Thus we have an expression for the canonical local height at infinity. We will only use a crude lower bound\footnote{In fact $\max_{\alpha\neq \beta: f(\alpha) = f(\beta) = 0} \frac{|f'(\alpha)|^{\frac{1}{2}}}{|\alpha - \beta|}\gg H(f)$. To see this, take $\alpha$ to be a root with maximal absolute value, and $\beta$ to be the root closest to $\alpha$. Since $5\alpha = \sum_{\rho\neq \alpha : f(\rho) = 0} \alpha - \rho$, it follows that $5|\alpha|\leq \sum_{\rho\neq \alpha : f(\rho) = 0} |\alpha - \rho|$. Since each $|\alpha - \rho|\leq 2|\alpha|$, at least three of the $\rho$ (namely, all the other roots besides $\alpha$ and $\beta$) must satisfy $|\alpha - \rho|\gg |\alpha|$. Since $|\alpha|\asymp H(f)$, it follows that $\frac{|f'(\alpha)|^{\frac{1}{2}}}{|\alpha - \beta|}\gg H(f)^{\frac{3}{2}}\cdot |\alpha - \beta|^{-\frac{1}{2}}$, and the claim follows.} for the last term, so let us get rid of it now:
\begin{lem}\label{the second lower bound on c beta}
Let $\alpha\neq \beta$ be roots of $f$. Then: $$c_\beta\geq 2\log{|\theta_{\chi_\beta + \chi_\infty + \chi_\alpha}(0)|} + \frac{1}{4}\log{|\Delta_f|} - 4 h(f) - O(1).$$
\end{lem}

\begin{proof}
Since $|\rho|\ll H(f)$ for all roots $\rho$ of $f$ (see Lemma \ref{sup norm of roots}), it follows that, for all $\rho, \rho'$ roots of $f$, $$|\rho - \rho'|\ll H(f)$$ as well. (Thus e.g.\ $|\alpha - \beta|\ll H(f)$.)

Now since $$|f'(\alpha)| = \prod_{\rho\neq \alpha} |\rho - \alpha|,$$ it follows that $$|\Delta_f| = \prod_{\rho\neq \rho'} |\rho - \rho'|^2 = |f'(\alpha)|^2\cdot \prod_{\rho,\rho'\neq \alpha, \rho\neq \rho'} |\rho - \rho'|^2\ll |f'(\alpha)|^2\cdot H(f)^{12}.$$ That is to say, $$|f'(\alpha)|\gg \frac{|\Delta_f|^{\frac{1}{2}}}{H(f)^6}.$$ Combining all these with Lemma \ref{first lower bound on c beta} gives the claim.
\end{proof}

Next we will show that we may ignore the contribution of the theta function. Let us first upper bound $\Xi_\chi(Z)$ --- at first uniformly, and then in the special case of $Z = A + \tau_f\cdot B$ with $A,B\in [-\eps,\eps]^{\times 2}$ we will obtain a significantly stronger bound.

\begin{lem}\label{the uniform upper bound on theta}
Let $\chi$ be a theta characteristic. Then: for any $Z\in \C^2$, $$|\Xi_\chi(Z)|\ll 1.$$
\end{lem}

\begin{proof}
By double periodicity, it suffices to take $Z\in [-\frac{1}{2}, \frac{1}{2}]^{\times 2} + \tau_f\cdot [-\frac{1}{2}, \frac{1}{2}]^{\times 2}$, a fundamental domain for $\Z^2 + \tau_f\cdot \Z^2$ in $\C^2$. We will bound $\Xi$ ``trivially'' --- i.e., via the triangle inequality. Note that, for such $Z$, $$|\Xi_{\chi}(Z)|\leq \sum_{\vec{n}\in \Z^2} e^{-\pi \left(\pair{\vec{n} + \vec{a}}{\Im{\tau_f}\cdot (\vec{n} + \vec{a})} + 2\pair{\vec{n} + \vec{a}}{\Im{Z}} + \pair{\Im{Z}}{(\Im{\tau_f})^{-1}\cdot \Im{Z}}\right)}.$$ Now the term in the exponent is just
\begin{align*}
&\pair{\vec{n} + \vec{a}}{\Im{\tau_f}\cdot (\vec{n} + \vec{a})} + 2\pair{\vec{n} + \vec{a}}{\Im{Z}} + \pair{\Im{Z}}{(\Im{\tau_f})^{-1}\cdot \Im{Z}} \\&= \pair{\vec{n} + \vec{a} + (\Im{\tau_f})^{-1}\cdot \Im{Z}}{\Im{\tau_f}\cdot \left(\vec{n} + \vec{a} + (\Im{\tau_f})^{-1}\cdot \Im{Z}\right)},
\end{align*} i.e.\ we have completed the square. Now since $\Im{\tau_2}\geq \Im{\tau_1}\geq 2\Im{\tau_{12}} > 0$ and $\Im{\tau_1}\geq \frac{\sqrt{3}}{2}$, we have that (by considering $\frac{\det{\Im{\tau}}}{\Tr{\, \Im{\tau}}}$) the eigenvalues of $\Im{\tau}$ are both $\gg 1$. It follows that $$\pair{v}{(\Im{\tau})\cdot v}\gg ||v||_2^2$$ for any $v\in \C^2$. Applying this to $\vec{n} + \vec{a} + (\Im{\tau})^{-1}\cdot \Im{Z}$, we find that this term is $$\gg ||\vec{n} + \vec{a} + (\Im{\tau})^{-1}\cdot \Im{Z}||_2^2.$$ Now since $Z\in [-\frac{1}{2},\frac{1}{2}]^{\times 2} + \tau_f\cdot [-\frac{1}{2}, \frac{1}{2}]^{\times 2}$, it follows that $(\Im{\tau})^{-1}\cdot \Im{Z}\in [-\frac{1}{2}, \frac{1}{2}]$, so that $$||\vec{n} + \vec{a} + (\Im{\tau})^{-1}\cdot \Im{Z}||_2^2\gg ||\vec{n}||_2^2 - O(1).$$

Therefore we have found that $$|\Xi_{\chi}(Z)|\ll \sum_{\vec{n}\in \Z^2} e^{-\Omega(||\vec{n}||_2^2)}\ll 1,$$ as desired.
\end{proof}

Having uniformly upper bounded the size of $\Xi_\chi(Z)$, we will now determine the size of $\Xi_\chi(Z)$ when $Z = A + \tau_f\cdot B$ and $A,B\in [-\eps,\eps]^{\times 2}$ --- in particular, we will also determine the size of the theta constants $\Xi_{\alpha,\beta}(0)$. To do this, we will simply use Proposition 7.6 of \cite{streng} (originally from \cite{klingen}), though we will modify it slightly by allowing the argument of the theta function to range in a very small neighbourhood about $0$.\footnote{While the extension to $Z\neq 0$ sufficiently close to $0$ is not necessary for our argument, if one is using the canonical height with an even characteristic (and partitioning the fundamental domain as we do in our argument) this generalized proposition is quite useful, and thus I have seen fit to include it.} By running the same analysis as is done in \cite{streng} (thus in \cite{klingen} --- just factor out the relevant exponential and observe that the negative-definite quadratic form in the exponent is strictly smaller away from the closest points to the origin, and then compute explicitly for those points), we find:
\begin{prop}[Cf.\ Proposition 7.6 in \cite{streng}.]\label{the better upper bound on theta near zero}
Let $Z = A + \tau_f\cdot B\in \C^2$ be such that $||A||, ||B||\ll \eps$, with $\eps\ll 1$ sufficiently small. Then:\footnote{Note that we have omitted the absolute values in the first four asymptotics: here, by $C\asymp 1$ we mean that there are positive absolute constants $\kappa > \kappa' > 0$ such that $|C - \kappa|\leq \kappa'$. This technically clashes with our definition of the symbol $\asymp$, hence this explanation. Note that the condition for $C$ implies it for $\Re{C}$.}
\begin{align*}
\theta_{0,0,0,0}(Z)&\asymp 1,\\
\theta_{0,0,\frac{1}{2},0}(Z)&\asymp 1,\\
\theta_{0,0,0,\frac{1}{2}}(Z)&\asymp 1,\\
\theta_{0,0,\frac{1}{2},\frac{1}{2}}(Z)&\asymp 1,\\
|\theta_{\frac{1}{2},0,0,0}(Z)|&\asymp e^{-\frac{\pi}{4}\Im{\tau_1} + O(\eps \Im{\tau_2})},\\
|\theta_{\frac{1}{2},0,0,\frac{1}{2}}(Z)|&\asymp e^{-\frac{\pi}{4}\Im{\tau_1} + O(\eps \Im{\tau_2})},\\
|\theta_{0,\frac{1}{2},0,0}(Z)|&\asymp e^{-\frac{\pi}{4}\Im{\tau_2} + O(\eps \Im{\tau_2})},\\
|\theta_{0,\frac{1}{2},\frac{1}{2},0}(Z)|&\asymp e^{-\frac{\pi}{4}\Im{\tau_2} + O(\eps \Im{\tau_2})},\\
|\theta_{\frac{1}{2},\frac{1}{2},0,0}(Z)|&\asymp e^{-\frac{\pi}{4}(\Im{\tau_1} + \Im{\tau_2} - 2\Im{\tau_{12}}) + O(\eps \Im{\tau_2})},\\
|\theta_{\frac{1}{2},\frac{1}{2},\frac{1}{2},\frac{1}{2}}(Z)|&\asymp \left|\cos{(\pi(Z_1 + Z_2))} e\left(\frac{\tau_{12}}{2}\right) - \cos{(\pi (Z_1 - Z_2))}\right|\cdot e^{-\frac{\pi}{4}(\Im{\tau_1} + \Im{\tau_2} - 2\Im{\tau_{12}})}.
\end{align*}
\end{prop}

Note, of course, that for any $Z$ we also have that $$-\log{|\Xi_\chi(Z)|} = -\log{|\theta_\chi(Z)|} + \pi \pair{\Im{Z}}{(\Im{\tau_f})^{-1}\cdot \Im{Z}}\geq -\log{|\theta_\chi(Z)|}$$ by positive-definiteness of $\Im{\tau_f}$, so for the purposes of lower bounds it suffices to just deal with the asymptotics of $\theta_\chi$ near $0$.

Now let us analyze the constants $c_\rho$.

\begin{lem}\label{the normal theta constant lower bound}
There is a root $\rho_*$ of $f$ such that, for all roots $\rho\neq \rho_*$ of $f$, $$c_\rho\geq \frac{1}{4}\log{|\Delta_f|} - 4 h(f) - O(1).$$
\end{lem}
\begin{proof}
Observe that, for each $\alpha$ such that $(\chi_\alpha)_a\neq \left(\frac{1}{2}, \frac{1}{2}\right)$, there is a $\beta$ such that $(\chi_\beta + \chi_\infty + \chi_\alpha)_a = (0,0)$. Indeed, $(\chi_\infty + \chi_\alpha)_a\neq (0,0)\in \frac{1}{2}\Z^2/\Z^2$, and so choosing an odd characteristic $\chi\neq \chi_\alpha, \chi_\infty$ with $\chi_a = (\chi_\infty + \chi_\alpha)_a$, the claim follows from Lemma \ref{the second lower bound on c beta} and Proposition \ref{the better upper bound on theta near zero}.
\end{proof}

We note here that this immediately implies a lower bound on the canonical local height at infinity for points very close to $[0,0,0,1]\in K_f(\C)$ in the Archimedean topology.
\begin{lem}\label{the lower bound on the canonical local height near infinity}
Let $\tilde{P} =: (\tilde{P}_1, \ldots, \tilde{P}_4)\in \tilde{K}_f(\Q)$ be such that $|\tilde{P}_{i+1}|\gg \delta^{-\delta^{-1}} H(f) |\tilde{P}_i|$ for $1\leq i\leq 3$. Then: $$\hat{\lambda}_\infty(\tilde{P}) = \lambda_\infty(\tilde{P}) + O(\delta h(f)).$$
\end{lem}

\begin{proof}
We note that $2^N \tilde{P}$ has the same property for $N\ll \delta^{-1}$, where we define, for $\tilde{R}\in \tilde{K}_f(\C)$, $$2\tilde{R} := (\delta_1(\tilde{R}), \ldots, \delta_4(\tilde{R})).$$ Indeed, by the explicit formulas and the dominance of $P_4$ in all expressions, we find that, if $|\tilde{R}_{i+1}|\geq C\cdot  \delta^{-\delta^{-1}} H(f) |\tilde{R}_i|$, then  $$(2\tilde{R})_i\asymp R_i\cdot R_4^3,$$ whence, for $1\leq i\leq 3$, $$|(2 \tilde{R})_{i+1}|\geq \delta\cdot C\cdot \delta^{-\delta^{-1}} H(f) |(2\tilde{R})_i|.$$ It follows that, for $N\asymp \delta^{-1}$, $$\frac{1}{4^N}\lambda_\infty(2^N \tilde{P}) = \lambda_\infty(\tilde{P}) + O(1).$$ (Here we have dropped the $N$ in $O_N(1)$ since $\delta$ is a (very, very small) constant.) Note also that, for $\rho$ a root of $f$, since $|\rho|\ll H(f)$ (see Lemma \ref{sup norm of roots}), $$\ell_\rho(2^N \tilde{P})\asymp (2^N \tilde{P})_3.$$ Now using Lemmas \ref{the formula for the canonical local height}, \ref{the uniform upper bound on theta}, and \ref{the normal theta constant lower bound} on $2^N \tilde{P}$, we find that $$\hat{\lambda}_\infty(\tilde{P})\geq \frac{1}{4^N}\lambda_\infty(2^N \tilde{P}) + O(4^{-N} h(f)),$$ whence the lower bound.

As for the upper bound, observe instead that $\lambda_\infty(2 R)\leq 4\lambda_\infty(R) + O(h(f))$ by the explicit formulas, and then apply the Tate telescoping series to get that $\hat{\lambda}_\infty(R)\leq \lambda_\infty(R) + O(h(f))$. Now use the above argument, except with this upper bound.
\end{proof}

We find as corollaries the case of $\kappa(P)$ with $P\in C_f(\Q)$ with large $x$-coordinate, as well as the case of $\kappa(P-Q)$ with $P\neq \pm Q\in C_f(\Q)$ both with large $x$-coordinates.
\begin{cor}\label{the lower bound on the canonical local height for big points}
Let $P\in C_f(\Q)$ with $|x(P)|\gg \delta^{-\delta^{-1}} H(f)$. Then: $$\hat{\lambda}_\infty(\kappa(P)) = \lambda_\infty(\kappa(P)) - O(\delta h(f)).$$
\end{cor}

\begin{proof}
Since $\kappa(P) = [0, 1, x(P), x(P)^2]$, the hypothesis of Lemma \ref{the lower bound on the canonical local height near infinity} follows.
\end{proof}

We also get a similarly strong statement for differences of two points with large $x$-coordinate.\footnote{In fact we have the slightly stronger bound $$\hat{\lambda}_\infty(P-Q)\geq \frac{1}{2}\max(\lambda_\infty(\kappa(P)),\lambda_\infty(\kappa(Q))) + \min(\lambda_\infty(\kappa(P)),\lambda_\infty(\kappa(Q))) - O(\delta h(f)),$$ which plays a role in bounding the number of \emph{integral} points on these curves --- e.g.\ for large points it results in a gap principle of shape $\cos{\theta}\leq \frac{1}{4}$, matching the Mumford gap principle for integral points observed by Helfgott and Helfgott-Venkatesh: one expects a right-hand side of $\frac{1}{g}$ for rational points, and $\frac{1}{2g}$ for integral points.}
\begin{lem}\label{the lower bound on the canonical local height for differences of big points}
Let $P\neq \pm Q\in C_f(\Q)$ with $|x(P)|, |x(Q)|\gg \delta^{-\delta^{-1}} H(f)$ with $y(P), y(Q)\geq 0$. Then: $$\hat{\lambda}_\infty(P-Q)\geq \frac{1}{2}\lambda_\infty(\kappa(P)) + \frac{1}{2}\lambda_\infty(\kappa(Q)) + h(f) - O(\delta h(f)).$$
\end{lem}

\begin{proof}
Write $P =: (X,Y)$ and $Q =: (x,y)$. By switching, we may assume without loss of generality that $|X|\geq |x|$. Note that, by hypothesis, since $X^5\sim f(X) = Y^2\geq 0$, it follows that $X\geq 0$, and similarly for $x$. Now let us examine the image of $P-Q$ in $K_f(\C)$ --- that is, the coordinates of $$\left(1, X+x , Xx, \frac{2a_5 + a_4(X+x) + 2a_3 Xx + a_2 Xx (X+x) + X^2 x^2 (X+x) + 2Yy}{(X-x)^2}\right).$$ The first three coordinates certainly satisfy the hypotheses of Lemma \ref{the lower bound on the canonical local height near infinity}, so let us show that the fourth coordinate is $\gg x^2 X$, which suffices since $|x|\gg \delta^{-\delta^{-1}} H(f)$. We bound the denominator by $|(X-x)^2|\ll X^2$ and note that the numerator is $X^2 x^2 (X+x) + 2Yy + O(H(f)^2 X^2 x)$. Now since $Y^2 = f(X) = X^5\cdot (1 + O(\delta))$, it follows that $Y\geq X^{\frac{5}{2}}\cdot (1 + O(\delta))$, and similarly for $y$. Thus the numerator is $\gg X^3 x^2$, as desired. Now Lemma \ref{the lower bound on the canonical local height near infinity} applies and we are done.
\end{proof}

It remains to analyze $c_\beta$ when $\chi_\beta = \left(\left(\frac{1}{2}, \frac{1}{2}\right), \left(\frac{1}{2}, 0\right)\right)$.

\begin{lem}\label{the bad theta constant lower bound}
Let $\beta$ be such that $\chi_\beta = \left(\left(\frac{1}{2}, \frac{1}{2}\right), \left(\frac{1}{2}, 0\right)\right)$. Then: $$c_\beta\geq -\frac{\pi}{2}\Im{\tau_1} + \frac{1}{4}\log{|\Delta_f|} - 4 h(f) - O(1).$$
\end{lem}
\begin{proof}
We will take $\alpha$ such that $\chi_\alpha = \left(\left(\frac{1}{2}, 0\right), \left(\frac{1}{2}, \frac{1}{2}\right)\right),$ so that $$\chi_\beta + \chi_\infty + \chi_\alpha = \left(\left(\frac{1}{2}, 0\right), \left(0, 0\right)\right).$$ It follows that, by Proposition \ref{the better upper bound on theta near zero}, $$|\theta_{\chi_\beta + \chi_\infty + \chi_\alpha}(0)|\asymp e^{-\frac{\pi}{4}\Im{\tau_1}},$$ as desired.
\end{proof}

We will next show that, when $Z = A + \tau_f\cdot B$ and $||A||, ||B||\ll \eps$, the extra $\frac{\pi}{2}\Im{\tau_1}$ term in the lower bound for $c_\beta$ will be cancelled by an improved upper bound on $\theta_{\chi_\beta}$.

\begin{lem}\label{the better upper bound on the bad theta near zero}
Let $\beta$ be such that $\chi_\beta = \left(\left(\frac{1}{2}, \frac{1}{2}\right), \left(\frac{1}{2}, 0\right)\right)$. Let $||A||, ||B||\ll \eps$ and $Z := A + \tau_f\cdot B$. Then: $$|\Xi_{\chi_\beta}(Z)|\ll e^{-(1 - O(\eps))\frac{\pi}{4}\Im{\tau_1}}.$$
\end{lem}

\begin{proof}
Observe that
\begin{align*}
|\Xi_{\chi_\beta}(Z)|&\leq \sum_{(n_1,n_2)\in \Z^2} e^{-\pi \pair{\columnvector{n_1 + \frac{1}{2}}{n_2 + \frac{1}{2}}}{\Im{\tau_f}\cdot \columnvector{n_1 + \frac{1}{2}}{n_2 + \frac{1}{2}}} - 2\pi \pair{\columnvector{n_1 + \frac{1}{2}}{n_2 + \frac{1}{2}}}{\Im{\tau_f}\cdot B} - \pi \pair{B}{\Im{\tau_f}\cdot B}} \\&= \sum_{(n_1,n_2)\in \Z^2} e^{-\pi \pair{\columnvector{n_1 + \frac{1}{2}}{n_2 + \frac{1}{2}} + B}{\Im{\tau_f}\cdot \left(\columnvector{n_1 + \frac{1}{2}}{n_2 + \frac{1}{2}} + B\right)}}.
\end{align*}

Now the quadratic form
\begin{align*}
&\pair{\columnvector{n_1 + \frac{1}{2}}{n_2 + \frac{1}{2}} + B}{\Im{\tau_f}\cdot \left(\columnvector{n_1 + \frac{1}{2}}{n_2 + \frac{1}{2}} + B\right)} \\&= (\Im{\tau_1} - \Im{\tau_{12}})(n_1 + \frac{1}{2} + B_1)^2 + (\Im{\tau_{12}})(n_1 + n_2 + 1 + B_1 + B_2)^2 + (\Im{\tau_2} - \Im{\tau_{12}})(n_2 + 
\frac{1}{2} + B_2)^2.
\\&\geq \frac{\Im{\tau_1}}{2}\left((n_1 + \frac{1}{2} + B_1)^2 + (n_2 + 
\frac{1}{2} + B_2)^2\right).
\end{align*}
First, note that this is always $\geq \frac{\Im{\tau_1}}{4}(1 - O(\eps))$. Moreover, once $||\vec{n}||\gg 1$, we find that it is $$\geq \frac{\Im{\tau_1}}{4}\left(1 + \Omega(||\vec{n}||_2^2) - O(\eps)\right),$$ from which the claim follows.
\end{proof}

We will next upper bound $\Im{\tau_1}$ via a study of Igusa invariants.

\begin{lem}\label{the upper bound on Im tau1}
$$\Im{\tau_1}\leq \frac{10}{\pi} h(f) - \frac{1}{3\pi}\log{|\Delta_f|} + O(1).$$
\end{lem}

\begin{proof}
Consider the reduced Igusa invariant (note: this notation differs from Igusa's original, but follows Streng \cite{streng}, at least up to normalizing (absolute) constants): $$i_3(f) := \frac{I_4^5}{I_{10}^2},$$ where, writing $\rho_1, \ldots, \rho_5$ for the roots of $f$, $$I_4 := \sum_{\sigma\in S_5} (\rho_{\sigma(1)} - \rho_{\sigma(2)})^2 (\rho_{\sigma(2)} - \rho_{\sigma(3)})^2 (\rho_{\sigma(3)} - \rho_{\sigma(1)})^2 (\rho_{\sigma(4)} - \rho_{\sigma(5)})^2\cdot \rho_{\sigma(4)}^2\cdot \rho_{\sigma(5)}^2$$ and $$I_{10} := \Delta_f.$$ (Here $I_4$ is the usual Igusa-Clebsch invariant of the binary sextic $\sum_{i=0}^5 a_i X^{5-i} Y^{i+1}$ --- recall that one of the branch points of the curve is at $\infty$, thus the zero at $Y=0$. We have computed the invariant via $(X,Y)\mapsto (Y,X)$, which replaces the roots with their inverses, and the usual definition for sextics.) Note that, by Lemma \ref{sup norm of roots}, $|I_4|\ll H(f)^{12}$. It follows therefore that $$|i_3(f)|\ll \frac{H(f)^{60}}{|\Delta_f|^2}.$$

Now, by \cite{streng} (originally in Igusa's \cite{igusa} --- see page 848 --- and apparently already computed in Bolza's \cite{bolza}), we have also that $$i_3(f) = \frac{\left(\sum_{\chi\text{ even}} \theta_\chi(0)^8\right)^5}{\left(\prod_{\chi\text{ even}} \theta_\chi(0)^2\right)^2}.$$ Applying Proposition \ref{the better upper bound on theta near zero} and assuming that $\Im{\tau_1}\gg 1$ (else we're done), we find that: $$|i_3(f)|\asymp \frac{e^{4\pi(\Im{\tau_1} + \Im{\tau_2} - \Im{\tau_{12}})}}{\min(1, |\tau_{12}|)}\gg \max(e^{6\pi \Im{\tau_1}}, e^{4\pi \Im{\tau_2}}).$$ Combining this with $$|i_3(f)|\ll \frac{H(f)^{60}}{|\Delta_f|^2}$$ gives the claim.
\end{proof}

After all this work, we have finally found that:
\begin{cor}\label{the uniform lower bound on the canonical local height}
Let $\tilde{P}\in \tilde{K}_f(\C)$ and $\rho$ a root of $f$. Then: $$\hat{\lambda}_\infty(\tilde{P})\geq \log{|\ell_\rho(\tilde{P})|} + \frac{5}{12}\log{|\Delta_f|} - 9 h(f) - O(\delta h(f)).$$
\end{cor}

\begin{proof}
We simply combine Lemmas \ref{the formula for the canonical local height}, \ref{the second lower bound on c beta}, \ref{the uniform upper bound on theta}, and \ref{the upper bound on Im tau1}.
\end{proof}

Moreover, as Lemma \ref{the better upper bound on the bad theta near zero} shows, with an extra hypothesis the above bound can be significantly improved. Namely, we have the following:
\begin{cor}\label{the better lower bound on the canonical local height near zero}
Let $Z$ be a set-theoretic section of $\C^2\surj \C^2/(\Z^2 + \tau_f\cdot \Z^2)\simeq J_f(\C)\surj K_f(\C)$. Let $\tilde{P}\in \tilde{K}_f(\C)$ be such that $Z(P) =: A + \tau_f\cdot B$ with $||A||, ||B||\ll \eps$. Let $\rho$ be a root of $f$. Then: $$\hat{\lambda}_\infty(\tilde{P})\geq \log{|\ell_\rho(\tilde{P})|} + \frac{1}{4}\log{|\Delta_f|} - 4 h(f) - O(\eps h(f)).$$
\end{cor}

\begin{proof}
Again combine Lemmas \ref{the formula for the canonical local height}, \ref{the second lower bound on c beta}, and \ref{the upper bound on Im tau1}, except now use Lemma \ref{the better upper bound on theta near zero}.
\end{proof}

We next claim that, in fact, these lower bounds are quite good for our purposes. To show this we will need to know something about the roots of $f$. Crucially, we will use that the $x^4$ coefficient --- the sum of the roots --- vanishes. This will ensure that there are two roots of $f$ that are $\gg H(f)$ away from each other (and both of size $\asymp H(f)$). First we must guarantee at least \emph{one} root of size $H(f)$ --- this exists for totally general reasons.

\begin{lem}\label{sup norm of roots}
$$\max_{f(\rho) = 0} |\rho|\asymp H(f).$$
\end{lem}

\begin{proof}
The upper bound follows from the fact that if $|z|\geq 100 H(f)$, then $|f(z)|\gg |z|^5$. The lower bound follows from the fact that, if, for all $\rho$ such that $f(\rho) = 0$, $|\rho|\leq \frac{H(f)}{100}$, then since $$a_i = (-1)^{5-i}\sum_{S\in {\mathrm{roots}(f)\choose i}} \prod_{\rho\in S} \rho,$$ we would have that $$|a_i|^{\frac{1}{i}} < \frac{H(f)}{2}$$ for all $i$, a contradiction.
\end{proof}

Next we find the two large roots that are far away from each other.

\begin{lem}\label{producing alpha star and beta star}
$\exists \alpha\neq \beta: f(\alpha) = f(\beta) = 0$, and: $$|\alpha|, |\beta|, |\alpha - \beta|\geq \frac{H(f)}{10^{10}}.$$
\end{lem}

\begin{proof}
Lemma \ref{sup norm of roots} produces an $\alpha$ with $f(\alpha) = 0$ and $|\alpha| > \frac{H(f)}{100}$. Now, if for all $\rho$ such that $f(\rho) = 0$, either $|\alpha - \rho| < \frac{|\alpha|}{100}$ or $|\rho| < \frac{|\alpha|}{100}$, then, writing $k := \#|\{\rho : f(\rho) = 0, |\alpha - \rho| < \frac{|\alpha|}{100}\}|\geq 1$, we would have that $$0 = \sum_{f(\rho) = 0} \rho = k\alpha + \sum_{f(\rho) = 0, |\alpha - \rho| < \frac{|\alpha|}{100}} \rho - \alpha + \sum_{f(\rho) = 0, |\rho| < \frac{|\alpha|}{100}} \rho,$$ and the first term dominates in size, a contradiction. Thus there is a root $\beta$ such that $|\beta| > \frac{|\alpha|}{100}$ and $|\alpha - \beta| > \frac{|\alpha|}{100}$, as desired.
\end{proof}
Now take $\alpha_*, \beta_*$ as in Lemma \ref{producing alpha star and beta star}.

Now, we will only ever apply our lower bound on $\hat{\lambda}_\infty$ to points of the form $\kappa(P)$ or $\kappa(P-Q)$ for $P\neq \pm Q\in C_f(\Q)$. For these points, we note that $$\ell_\rho(0,1,x,x^2) = x-\rho,$$ and $$\ell_\rho\left(1,X+x,Xx,\frac{2a_5 + a_4(X+x) + 2a_3 Xx + a_2 Xx (X+x) + X^2 x^2 (X+x) + 2Yy}{(X-x)^2}\right) = (X - \rho)(x - \rho).$$ Thus, $$|\ell_{\alpha_*}(0,1,x,x^2) - \ell_{\beta_*}(0,1,x,x^2)| = |\alpha_* - \beta_*|\gg H(f).$$ This implies (using Lemma \ref{the uniform lower bound on the canonical local height}) that:
\begin{cor}\label{just point lower bound}
Let $P\in C_f(\Q)$. Then: $$\hat{\lambda}_\infty(\tilde{P})\geq \max(\frac{1}{2}\lambda_\infty(P), h(f)) + \frac{5}{12}\log{|\Delta_f|} - 9 h(f) - O(\delta h(f)).$$
\end{cor}
(We have used the simple bound $|x - \rho|\gg |x|$ if $|x|\gg \delta^{-1} H(f)$ and $\rho$ is a root of $f$.)

Similarly, for $\hat{\lambda}_\infty(P-Q)$, we get (using Corollary \ref{the better lower bound on the canonical local height near zero}):
\begin{cor}\label{difference lower bound}
Let $P\neq \pm Q\in C_f(\Q)$ be such that $x(P)$ and $x(Q)$ are closest to the same element of $\{\alpha_*, \beta_*\}$ and such that the coset $\Psi_f^{-1}(P-Q)\subseteq \C^2$ of $\Z^2 + \tau_f\cdot \Z^2$ contains an element $A + \tau_f\cdot B$ with $||A||, ||B||\ll \eps$. Then: $$\hat{\lambda}_\infty(P-Q)\geq \frac{1}{4}\log{|\Delta_f|} - 2 h(f) - O(\eps h(f)).$$
\end{cor}

\begin{proof}
Note that, in general, $\max(|x - \alpha_*|, |x - \beta_*|)\gg H(f)$ (since $H(f)\ll |\alpha_* - \beta_*|\leq |x - \alpha_*| + |x - \beta_*|$). Without loss of generality, let us suppose that $\beta_*$ is the closest of $\{\alpha_*, \beta_*\}$ to both $x(P) =: X$ and $x(Q) =: x$. Since $$\ell_{\alpha_*}\left(1,X+x,Xx,\frac{2a_5 + a_4(X+x) + 2a_3 Xx + a_2 Xx (X+x) + X^2 x^2 (X+x) - 2Yy}{(X-x)^2}\right) = (X - \alpha_*)(x - \alpha_*)\gg H(f)^2,$$ the result follows from Corollary \ref{the better lower bound on the canonical local height near zero}.
\end{proof}

Having completed our analysis of $\hat{\lambda}_\infty$, let us return to the postponed analysis of $\hat{h}(P-Q)$ for points without an unusually large $x$-coordinate.

\subsubsection{Lower bounds on $\hat{h}(P-Q)$.\label{lower bound on h hat of P-Q}}
Next we will lower bound, for $P$ and $Q$ non-small points, $\hat{h}(P-Q)$. Since our lower bound on the height of a non-small point is so large, we will not need to do any delicate analysis for the lower bound (much like the upper bound). The only difficulty will be in guaranteeing that Corollary \ref{difference lower bound} is applicable if the points do not have big $x$-coordinates. To do this we will, of course, introduce a further partition of our points.

But first, let us finish off the case of points with large $x$-coordinate.
\begin{lem}\label{the lower bound for differences of big points}
Let $P\neq \pm Q\in \II_f$ be such that $|x(P)|, |x(Q)|\gg \delta^{-\delta^{-1}} H(f)$ and $y(P), y(Q)\geq 0$. Then: $$\hat{h}(P-Q)\geq \frac{1}{2} h_K(P) + \frac{1}{2} h_K(Q) - \frac{17}{3} h(f) - O(\delta h(f)).$$
\end{lem}

\begin{proof}
Write $x(P) =: X =: \frac{S}{D^2}$ and $x(Q) =: x =: \frac{s}{d^2}$, both in lowest terms, so that we are analyzing the point $$\left(1, X + x, Xx, \frac{2a_5 + a_4(X+x) + 2a_3 Xx + a_2 Xx (X+x) + X^2 x^2 (X+x) + 2Yy}{(X-x)^2}\right).$$ Recall that we have already lower bounded $\hat{\lambda}_\infty(P-Q)$ --- namely, Lemma \ref{the lower bound on the canonical local height for differences of big points} tells us that: $$\hat{\lambda}_\infty(P-Q)\geq \frac{1}{2}\lambda_\infty(\kappa(P)) + \frac{1}{2}\lambda_\infty(\kappa(Q)) + h(f) - O(\delta h(f)).$$ To lower bound $\hat{\lambda}_p(P-Q)$, we will simply lower bound $\lambda_p(P-Q)$ and apply Corollary \ref{stoll's bound}. To lower bound $$\lambda_p(P-Q) := \log{p}\cdot \max_i(-v_p(\kappa(P-Q)_i)),$$ we first observe that, of course, since the first coordinate is $1$, it follows that $\lambda_p(P-Q)\geq 0$ always. Next, if $\max(-v_p(X), -v_p(x)) > 0$, it follows (by breaking into cases based on whether or not both are positive) that $$\max(-v_p(X + x), -v_p(Xx))\geq 2v_p(D) + 2v_p(d).$$ Hence we have found that $$\sum_p \lambda_p(P-Q)\geq 2\log{|d|} + 2\log{|D|},$$ which is to say that $$\sum_p \lambda_p(P-Q)\geq \sum_p \frac{1}{2}\lambda_p(\kappa(P)) + \frac{1}{2}\lambda_p(\kappa(Q)).$$ Applying Stoll \cite{stollii} (i.e.\ Corollary \ref{stoll's bound}), we find that $$\sum_p \hat{\lambda}_p(P-Q)\geq \sum_p \frac{1}{2}\lambda_p(\kappa(P)) + \frac{1}{2}\lambda_p(\kappa(Q)) - \frac{1}{3}\log{|\Delta_f|}.$$ Combining this with the lower bound on $\hat{\lambda}_\infty(P-Q)$ and the fact that $|\Delta_f|\ll H(f)^{20}$ gives the claim.
\end{proof}

Write $$\Gunct := \left[-\frac{1}{2}, \frac{1}{2}\right]^{\times 2} + \tau_f\cdot \left[-\frac{1}{2},\frac{1}{2}\right]^{\times 2}.$$ Write $$\Gunct^{(i_1, \ldots, i_4)} := \left(\left[\frac{i_1}{2N}, \frac{i_1+1}{2N}\right]\times \left[\frac{i_2}{2N}, \frac{i_2+1}{2N}\right]\right) + \tau_f\cdot \left(\left[\frac{i_3}{2N}, \frac{i_3+1}{2N}\right]\times \left[\frac{i_4}{2N}, \frac{i_4+1}{2N}\right]\right),$$ where $N\asymp \delta^{-1}$ (thus this is a partition into $O(1)$ parts, since $\delta\gg 1$). Note that $$\Gunct = \bigcup_{i_1 = -N}^N\bigcup_{i_2 = -N}^N\bigcup_{i_3 = -N}^N\bigcup_{i_4 = -N}^N \Gunct^{(i_1,i_2,i_3,i_4)}.$$ Now let $$Z: J_f(\C)\to \Gunct$$ be a set-theoretic section (observe that the map $\Gunct\to \C^2/(\Z^2 + \tau_f\cdot \Z^2)\simeq J_f(\C)$ is surjective). Next let $$\II_f^{(i_1,i_2,i_3,i_4)} := \II_f\cap Z^{-1}(\Gunct^{(i_1,i_2,i_3,i_4)}).$$ (Similarly with decorations such as $\up, \down, \bullet$ added, and for $\III_f^{(i_1, \ldots, i_4)}$.) Thus if $P,Q\in \II_f^{(i_1,i_2,i_3,i_4)}$, we have that $$Z(P)-Z(Q) = A + \tau_f\cdot B$$ with $$||A||, ||B||\ll \delta.$$ Finally, recall (via Lemma \ref{producing alpha star and beta star}) that we chose two roots $\alpha_*, \beta_*$ of $f$ such that $$|\alpha_*|, |\beta_*|, |\alpha_* - \beta_*|\gg H(f).$$ So let $$\II_f^{\alpha_*} := \{P\in \II_f : |x(P) - \alpha_*|\leq |x(P) - \beta_*|\}$$ and, similarly, $$\II_f^{\beta_*} := \{P\in \II_f : |x(P) - \beta_*|\leq |x(P) - \alpha_*|\}.$$ That is, $\II_f^{\alpha_*}$ is the set of points of $\II_f$ whose $x$-coordinates are closest to $\alpha_*$, and similarly for $\II_f^{\beta_*}$. We similarly define $\III_f^{\alpha_*}$ and $\III_f^{\beta_*}$. Finally, for $\rho\in \{\alpha_*, \beta_*\}$, define $$\II_f^{(i_1,i_2,i_3,i_4), \rho} := \II_f^{(i_1,i_2,i_3,i_4)}\cap \II_f^\rho,$$ and similarly for $\III_f$, and all other decorations.

Having suitably refined our partition, let us now deal with points whose $x$-coordinate is not so large.
\begin{lem}\label{the lower bound for differences of normal points}
Let $\rho\in \{\alpha_*, \beta_*\}$. Let $P\neq \pm Q\in \II_f^{(i_1,i_2,i_3,i_4), \rho}$ be such that $|x(P)|, |x(Q)|\ll \delta^{-\delta^{-1}} H(f)$. Then: $$\hat{h}(P-Q)\geq \frac{1}{2} h_K(P) + \frac{1}{2} h_K(Q) - \frac{17}{3} h(f) - O(\delta h(f)).$$
\end{lem}
Of course the same result holds for $\III_f^{(i_1, \ldots, i_4), \rho}$ as well.

\begin{proof}
The argument is precisely the same as for Lemma \ref{the lower bound for differences of big points}, except that for the lower bound on $\hat{\lambda}_\infty$ we use Corollary \ref{difference lower bound} (--- this is where we use the partition), and we note that $$\sum_p \frac{1}{2}\lambda_p(\kappa(P)) + \frac{1}{2}\lambda_p(\kappa(Q))\geq \frac{1}{2}h_K(P) + \frac{1}{2}h_K(Q) - 2h(f).$$ Thus we find that:
\begin{align*}
\hat{h}(P-Q)&\geq \frac{1}{2}h_K(P) + \frac{1}{2}h_K(Q) + \frac{1}{4}\log{|\Delta_f|} - 4 h(f) - \frac{1}{3}\log{|\Delta_f|} + O(\delta h(f))
\\&= \frac{1}{2}h_K(P) + \frac{1}{2}h_K(Q) - \frac{17}{3} h(f) + O(\delta h(f))
\end{align*}
as desired.
\end{proof}

Now that we have adequately (thanks to our strong lower bounds on the heights of medium points) lower bounded the canonical height of a difference, we may finally prove the claimed gap principles. To do so we will need some final preparatory work to ensure the points we consider are indeed suitably close (so that they may be seen to repulse). That is to say, we will have to ensure that their canonical heights are comparable --- note that, at the moment, we may only a priori ensure that their \emph{na{\" i}ve} heights are comparable. But it turns out we have done enough to guarantee the former as well.

\subsubsection{Partitioning based on the size of $\hat{h}$.}

The following shows that we may now guarantee that $h_K$ and $\hat{h}$ are within a multiplicative constant in our range.

\begin{lem}\label{the heights are comparable}
Let $P\in C_f(\Q) - \I_f$. Then: $$\hat{h}(P)\asymp h_K(P).$$
\end{lem}

\begin{proof}
Write $$\hat{h}(P) = h_K(P) + \left[\hat{\lambda}_\infty(\kappa(P)) - \lambda_\infty(\kappa(P))\right] + \sum_p \left[\hat{\lambda}_p(\kappa(P)) - \lambda_p(\kappa(P))\right].$$ By Corollary \ref{stoll's bound}, $$\sum_p |\hat{\lambda}_p(\kappa(P)) - \lambda_p(\kappa(P))|\leq \frac{1}{3}\log{|\Delta_f|} + O(1)\leq \frac{20}{3} h(f) + O(1).$$

For the upper bound, note that, from the doubling formulas, evidently $$\lambda_\infty(2R) - 4\lambda_\infty(R)\leq 12 h(f) + O(1),$$ so that\footnote{In fact it is rather easy to get that, for $R\in K_f(\Q)$, $\hat{\lambda}_\infty(R) - \lambda_\infty(R)\leq 3 h(f) + O(\eps h(f))$ by following the same analysis done in Lemma \ref{the upper bound for sums of normal points}. Indeed, one finds that $|(2^N R)_i|\ll \max_{\alpha_1 + \cdots + \alpha_4 = 4^N} H(f)^{4^{N+1} - 4 + i - \alpha_1 - \cdots - 4\alpha_4} H_K(P)^{4^N}$ and concludes in the same way. Note that this argument gives a bound of $\hat{\lambda}_\infty(\kappa(P)) - \lambda_\infty(\kappa(P))\leq 2 h(f) + O(\eps h(f))$, since $\kappa(P)_1 = 0$, and so $\alpha_1 = 0$ is forced (whence the maximum is achieved at $\alpha_2 = 4^N$, rather than $\alpha_1 = 4^N$).} $$\hat{\lambda}_\infty(\kappa(P)) - \lambda_\infty(\kappa(P))\leq 4 h(f) + O(1),$$ and hence $$\hat{h}(P)\leq h_K(P) + O(h(f)),$$ which is enough since $h(P)\gg h(f)$ since $P$ is not small.

The lower bound is a bit more difficult, and for it we will break into cases.

If $|x(P)| > \delta^{-\delta^{-1}} H(f)$, then $h(P)\geq (c_\up - \delta) h(f) = \left(\frac{25}{3} - \delta\right) h(f)$, and thus $h_K(P) = 2h(P) > (\frac{50}{3} - 2\delta) h(f)$. Moreover we have seen (Lemma \ref{the lower bound on the canonical local height for big points}) that $$\hat{\lambda}_\infty(\kappa(P))\geq \lambda_\infty(\kappa(P)) - O(\delta h(f)).$$ Therefore we find that, in this case, $$\hat{h}(P)\geq h_K(P) - \frac{20}{3} h(f) - O(\delta h(f)).$$ Since $h_K(P) > \left(\frac{50}{3} - \delta\right) h(f)$ this is enough.

If $|x(P)| < \delta^{-\delta^{-1}} H(f)$, then $h_K(P)\geq 2(c_\down - \delta) h(f) = (16 - 2\delta) h(f)$. Also since $\lambda_\infty(\kappa(P))\geq \lambda_\infty(\kappa(P)) + \frac{5}{12}\log{|\Delta_f|} - 9 h(f) - O(\delta h(f))$ by Corollary \ref{just point lower bound}, we may follow the argument in the above case to find that 
\begin{align*}
\hat{h}(P)&\geq h_K(P) + \frac{5}{12}\log{|\Delta_f|} - 10 h(f) - \frac{1}{3}\log{|\Delta_f|} - O(\delta h(f))
\\&= h_K(P) + \frac{1}{12}\log{|\Delta_f|} - 10 h(f) - O(\delta h(f))
\\&\geq h_K(P) - 10 h(f) + O(\delta h(f))
\end{align*}
which is again enough. This completes the argument.
\end{proof}

Thus Lemma \ref{the heights are comparable} furnishes us with constants $\mu, \nu$ with $\mu\asymp 1, \nu\asymp \delta^{-\delta^{-1}}$ such that, for all $P\in \II_f$, we have that both $\hat{h}(P)\in [\mu^{-1} h_K(P), \mu h_K(P)]$, and $h_K(P)\in [\nu^{-1} h(f), \nu h(f)]$. In order to break into points with very (multiplicatively) close heights (and canonical heights, since a priori they may still be wildly different), we will partition as follows. Note that this is precisely the situation in which we have a chance of seeing a repulsion phenomenon --- our points from now on will be close in size in the Mordell-Weil lattice.

Note that $$[\mu^{-1}, \mu]\subseteq \Cup_{i=-O(\delta^{-1})}^{O(\delta^{-1})} [(1+\delta)^i, (1+\delta)^{i+1}],$$ and similarly for $[\nu^{-1}, \nu]$ (except with the bounds on the union changed to $O(\delta^{-2}\log{\delta^{-1}})$). Define the following partition of $\II_f$ into $\delta^{-O(1)}$ many pieces: $$\II_f^{[i,j]} := \{P\in \II_f : \hat{h}(P)\in [(1+\delta)^i h_K(P), (1+\delta)^{i+1} h_K(P)]\text{ and } h_K(P)\in [(1+\delta)^j h(f), (1+\delta)^{j+1} h(f)]\},$$ and similarly with all other decorations added --- e.g., $$\II_f^{\up, (i_1, i_2, i_3, i_4), \rho, [i,j]} := \II_f^{\up, (i_1, i_2, i_3, i_4), \rho}\cap \II_f^{[i,j]}.$$ We also define $\III_f^{[[i]]}$, etc.\ (thus also e.g.\ $\III_f^{\bullet, (i_1, \ldots, i_4), \rho, [[i]]}$) in a similar way, except without the second condition --- that is, we only impose that $\hat{h}(P)\in [(1+\delta)^i h_K(P), (1+\delta)^{i+1} h_K(P)]$.

Note that, by construction, if $P,Q\in \II_f^{[i,j]}$, then $$\left|\frac{h_K(P)}{h_K(Q)} - 1\right|, \left|\frac{\hat{h}(P)}{\hat{h}(Q)} - 1\right|\ll \delta.$$ This will allow us to replace e.g.\ $h_K(Q)$ by $h_K(P)$ (and the same for $\hat{h}$) in the expression for $\cos{\theta_{P,Q}}$ without incurring a nontrivial error. Having defined this partition, let us now finally prove the promised gap principles for $\hat{h}$.

\subsubsection{The explicit gap principles via switching.}

Let us now, finally, prove the gap principles. First, we will deal with the case of points with large $x$-coordinate.
\begin{lem}\label{mumford gap principle for big points}
Let $P\neq \pm Q\in \II_f^{\up, (i_1, \ldots, i_4), \rho, [i,j]}$ and such that $y(P), y(Q)\geq 0$. Then: $$\cos{\theta_{P,Q}}\leq \frac{39}{59} + O(\delta)\leq 0.6334.$$
\end{lem}

\begin{proof}
We break into cases based on whether $\hat{h}(P)\geq h_K(P) - \frac{5}{3} h(f)$ or not. If $\hat{h}(P)\geq h_K(P) - \frac{5}{3} h(f)$, then we use the formula $$\cos{\theta_{P,Q}} = \frac{\hat{h}(P+Q) - \hat{h}(P) - \hat{h}(Q)}{2\sqrt{\hat{h}(P)\hat{h}(Q)}}.$$ Now since $\hat{h}(Q) = \hat{h}(P)\cdot (1 + O(\delta))$, we find that: $$\cos{\theta_{P,Q}}\leq \frac{\hat{h}(P+Q)}{2\hat{h}(P)} - 1 - O(\delta).$$ We now apply our hypothesis for this case, namely that $\hat{h}(P)\geq h_K(P) - \frac{5}{3} h(f)$, to find that $$\cos{\theta_{P,Q}}\leq \frac{\hat{h}(P+Q)}{2 h_K(P) - \frac{10}{3} h(f)} - 1 - O(\delta).$$ Now we apply Lemma \ref{the upper bound for sums of big points} and the fact that $h_K(Q) = h_K(P)\cdot (1 + O(\delta))$ to find that $$\cos{\theta_{P,Q}}\leq \frac{1}{2} + \frac{6 h(f)}{3 h_K(P) - 5 h(f)} + O(\delta).$$ Since $h_K(P)\geq \frac{50}{3} h(f) - O(\delta h(f))$, we find that $$\cos{\theta_{P,Q}}\leq \frac{19}{30} + O(\delta),$$ finishing this case.

Thus we are left with the case of $\hat{h}(P) < h_K(P) - \frac{5}{3} h(f)$, for which we use the formula $$\cos{\theta_{P,Q}} = \frac{\hat{h}(P) + \hat{h}(Q) - \hat{h}(P-Q)}{2\sqrt{\hat{h}(P)\hat{h}(Q)}}.$$ Again this is simply $$\cos{\theta_{P,Q}}\leq 1 - \frac{\hat{h}(P-Q)}{2\hat{h}(P)} + O(\delta),$$ and now we use our hypothesis to get that $$\cos{\theta_{P,Q}}\leq 1 - \frac{\hat{h}(P-Q)}{2 h_K(P) - \frac{10}{3} h(f)}.$$ But now Lemma \ref{the lower bound for differences of big points} tells us that therefore $$\cos{\theta_{P,Q}}\leq \frac{1}{2} + \frac{6 h(f)}{3 h_K(P) - 5 h(f)} + O(\delta),$$ which is the same expression we got in the previous case, QED.
\end{proof}

Now let us prove the gap principle for points whose $x$-coordinate is not so large.
\begin{lem}\label{mumford gap principle for normal points}
Let $P\neq \pm Q\in \II_f^{(i_1, i_2, i_3, i_4), \rho, [i,j]}$ be such that $|x(P)|, |x(Q)|\ll \delta^{-\delta^{-1}} H(f)$. Then: $$\cos{\theta_{P,Q}}\leq \frac{64}{95} + O(\delta)\leq 0.6737.$$
\end{lem}

\begin{proof}
The proof is the same, except instead we use Lemmas \ref{the upper bound for sums of normal points} and \ref{the lower bound for differences of normal points}, and we split into cases based on whether $\hat{P}\geq h_K(P) - \frac{1}{6} h(f)$ or not. In both cases we get that $$\cos{\theta_{P,Q}}\leq \frac{1}{2} + \frac{33 h(f)}{12 h_K(P) - 2 h(f)} + O(\delta).$$ Using $h_K(P)\geq 16 h(f) - O(\delta h(f))$ gives the claim.
\end{proof}

Thus we have proved our desired gap principles.

\subsubsection{Concluding via the sphere-packing argument.}

We will need a small sphere-packing preliminary, as used in Helfgott-Venkatesh \cite{helfgottvenkatesh}. To establish the result we will use the following theorem of Kabatiansky-Levenshtein \cite{kabatianskylevenshtein}:
\begin{thm}[Kabatiansky-Levenshtein \cite{kabatianskylevenshtein}]\label{kabatiansky levenshtein theorem}
Let $A\subseteq S^{n-1}\subseteq \R^n$ be such that for all $v\neq w\in A$, $$\cos{\theta_{v,w}}\leq \eta.$$ Then:
\begin{align*}
\#|A|\ll \exp\left(n\cdot \left[\vphantom{\frac{1}{1}}\hspace{-.1in}\right.\right.&\left.\left.\left(\frac{1 + \sin(\arccos(\eta))}{2\sin(\arccos(\eta))}\right)\log{\left(\frac{1 + \sin(\arccos(\eta))}{2\sin(\arccos(\eta))}\right)} \right.\right.\\&\left.\left.- \left(\frac{1 - \sin(\arccos(\eta))}{2\sin(\arccos(\eta))}\right)\log{\left(\frac{1 - \sin(\arccos(\eta))}{2\sin(\arccos(\eta))}\right)}\vphantom{\frac{1}{1}}\hspace{.0in}\right]\right).
\end{align*}
\end{thm}

From this we will derive the following immediate corollary.
\begin{cor}\label{kabatiansky levenshtein corollary}
Let $A\subseteq \R^n$ be such that for all $v\neq w\in A$, $$\cos{\theta_{v,w}}\leq \eta.$$ Then:
\begin{align*}
\#|A|\ll \exp\left(n\cdot \left[\vphantom{\frac{1}{1}}\hspace{-.1in}\right.\right.&\left.\left.\left(\frac{1 + \sin(\arccos(\eta))}{2\sin(\arccos(\eta))}\right)\log{\left(\frac{1 + \sin(\arccos(\eta))}{2\sin(\arccos(\eta))}\right)} \right.\right.\\&\left.\left.- \left(\frac{1 - \sin(\arccos(\eta))}{2\sin(\arccos(\eta))}\right)\log{\left(\frac{1 - \sin(\arccos(\eta))}{2\sin(\arccos(\eta))}\right)}\vphantom{\frac{1}{1}}\hspace{.0in}\right]\right).
\end{align*}
\end{cor}

\begin{proof}
Without loss of generality, $0\not\in A$ (since removing one element does not affect the bound) and $\eta\neq 1$ (since otherwise the right-hand side is infinite). The map $\phi: A\to S^{n-1}$ via $v\mapsto \frac{v}{|v|}$ is then injective, since otherwise if $v\neq w\in A$ were to map to the same point, then it would be the case that $\cos{\theta_{v,w}} = \pair{\frac{v}{|v|}}{\frac{w}{|w|}} = 1 > \eta$, a contradiction. Now just apply Theorem \ref{kabatiansky levenshtein theorem} to $\phi(A)$.
\end{proof}

From this we may immediately bound the number of medium points on our curves.

\begin{lem}\label{the medium point bound}
$$\#|\II_f|\ll 1.645^{\rank(J_f(\Q))}.$$
\end{lem}

\begin{proof}
We may assume without loss of generality that, for all $P\in \II_f$, $y(P)\geq 0$.\footnote{Of course this is not literally true, but the resulting bound will only be worsened by a factor of $2$ to make up for this assumption.} Since $$\II_f = \bigcup_{\rho\in \{\alpha_*, \beta_*\}}\bigcup_{?\in \{\up,\bullet,\down\}}\bigcup_{i_1 = 0}^{O(\delta^{-1})}\cdots \bigcup_{i_4 = 0}^{O(\delta^{-1})}\bigcup_{i=-O(\delta^{-1})}^{O(\delta^{-1})}\bigcup_{j=-\delta^{-O(1)}}^{\delta^{-O(1)}} \II_f^{?,(i_1,\ldots,i_4),\rho,[i,j]}$$ is a partition into $\delta^{-O(1)} = O(1)$ parts, it suffices to prove this bound for each of the parts of the partition. But for each $P\neq Q\in \II_f^{?,(i_1,\ldots,i_4),\rho,[i,j]}$ (the bound is trivial when $Q = -P$), we have proven that $\cos{\theta_{P,Q}}\leq 0.6737$ (and in fact the bound is even a bit better when $? = \up$). It follows then from Corollary \ref{kabatiansky levenshtein corollary} that $$\#|\II_f^{?,(i_1,\ldots,i_4),\rho,[i,j]}|\ll 1.645^{\rank(J_f(\Q))},$$ as desired.
\end{proof}

This completes the medium point analysis!

\subsection{Large points.}

Given all our gap principles above, the large point analysis is in fact quite quick. We recall the theorem of Bombieri-Vojta (with explicit determination of implicit constants (absolutely crucial for this work!) done by Bombieri-Granville-Pintz) once again:
\begin{thm}[Bombieri-Vojta, Bombieri-Granville-Pintz]\label{bombieri vojta}
Let $\delta > 0$. Let $f\in \Z[x]$ be a monic polynomial of degree $2g+1\geq 5$ with no repeated roots, and let $C_f: y^2 = f(x)$. Let $\alpha\in \left(\frac{1}{\sqrt{g}}, 1\right)$. Let $P\neq Q\in C_f(\Q)$ be such that $\hat{h}(P)\geq \delta^{-\frac{1}{2}}\hat{h}(Q)\geq \delta^{-1} h(f)$. Then, once $\delta\ll_\alpha 1$, we have that: $$\cos{\theta_{P,Q}}\leq \alpha.$$
\end{thm}

Given this, the going is quite easy. The only thing to note is the following improvement on the above gap principles via the largeness of our points.
\begin{lem}\label{mumford gap principle for big large points}
Let $P\neq \pm Q\in \III_f^{\up, (i_1, \ldots, i_4), \rho, [[i]]}$ be such that $\hat{h}(P) = \hat{h}(Q)\cdot (1 + O(\delta))$ and $y(P), y(Q)\geq 0$. Then: $$\cos{\theta_{P,Q}}\leq \frac{1}{2} + O(\delta).$$
\end{lem}

\begin{proof}
The proof is the same as the argument in Lemma \ref{mumford gap principle for big points}, except that the term $\frac{6 h(f)}{3 h_K(P) - 5 h(f)}\ll \delta$.
\end{proof}

Similarly, of course, for points without a large $x$-coordinate.
\begin{lem}\label{mumford gap principle for normal large points}
Let $P\neq \pm Q\in \III_f^{\bullet, (i_1, \ldots, i_4), \rho, [[i]]}\cup \III_f^{\down, (i_1, \ldots, i_4), \rho, [[i]]}$ be such that $\hat{h}(P) = \hat{h}(Q)\cdot (1 + O(\delta))$. Then: $$\cos{\theta_{P,Q}}\leq \frac{1}{2} + O(\delta).$$
\end{lem}

\begin{proof}
Again, the proof is the same as the argument in Lemma \ref{mumford gap principle for normal points}, except now the term $\frac{33 h(f)}{12 h_K(P) - 2 h(f)}\ll \delta$.
\end{proof}

Having done this, we may finish the proof of Theorem \ref{the big kahuna}.\footnote{Here we give the argument obtaining the bound $\ll 1.888^{\rank(J_f(\Q))}$, despite writing above that we get $\ll 1.872^{\rank(J_f(\Q))}$, because it has a pretty endgame. Lemma \ref{the optimized large point bound} has the proof of the slightly better bound.}
\begin{lem}\label{the large point bound}
$$\#|\III_f|\ll 1.888^{\rank(J_f(\Q))}.$$
\end{lem}

\begin{proof}
We may assume without loss of generality that, for all $P\in \III_f$, $y(P)\geq 0$.\footnote{Again, this is not literally true, but the resulting bound will only worsen by a factor of $2$ since these are at least half of all the points in $\III_f$.}
Let $\alpha := \frac{3}{4} + \delta^{\frac{1}{2}}$. Note that Theorem \ref{bombieri vojta} applies to $\alpha$ since $\frac{1}{\sqrt{2}} < 0.7072$. Let $S\subseteq \III_f$ be maximal such that: for all $P\neq Q\in S$, we have that $$\cos{\theta_{P,Q}}\leq \alpha.$$ Of course, by Kabatiansky-Levenshtein (Corollary \ref{kabatiansky levenshtein corollary}), once $\delta\ll 1$, it follows that $$\#|S|\ll 1.888^{\rank(J_f(\Q))}.$$ Thus it suffices to show that $$\#|\III_f|\ll \#|S|.$$

Now observe that, by maximality, for all $P\in \III_f$, there is a $Q\in S$ such that $$\cos{\theta_{P,Q}} > \alpha.$$ Of course by Theorem \ref{bombieri vojta} it follows that $\hat{h}(P)\asymp \hat{h}(Q)$. Thus, it follows that:
$$\III_f = \bigcup_{Q\in S}\bigcup_{|k|\ll \delta^{-O(1)}} \III_f^{(Q,k)},$$ where $$\III_f^{(Q,k)} := \{P\in \III_f : \cos{\theta_{P,Q}} > \alpha, \frac{\hat{h}(P)}{\hat{h}(Q)}\in [(1 + \delta)^k, (1 + \delta)^{k+1}]\}.$$ Since this partition is into $\ll \delta^{-O(1)}\#|S|$ parts, it suffices to show that each $\#|\III_f^{(Q,k)}|\ll 1$. As usual, we may restrict further to $$\III_f^{?,(Q,k),(i_1,\ldots,i_4),\rho,[[i]]} := \III_f^{(Q,k)}\cap \III_f^{?,(i_1,\ldots,i_4),\rho,[[i]]},$$ since this only worsens our bound by a factor of $O(1)$.

To do this we will first show that, after removing at most one element from $\III_f^{?,(Q,k),(i_1,\ldots,i_4),\rho,[[i]]}$, for each remaining $P\neq P'\in \III_f^{?,(Q,k),(i_1,\ldots,i_4),\rho,[[i]]}$, we in fact have that $$\cos{\theta_{P,P'}}\leq -\Omega(\delta^{\frac{1}{2}}).$$

We show this as follows. For each $P\in \III_f^{?,(Q,k),(i_1,\ldots,i_4),\rho,[[i]]}$, let $v_P := \frac{P}{|P|} - \frac{Q}{|Q|}$, where $$\frac{R}{|R|} := \frac{1}{\sqrt{\hat{h}(R)}}\otimes R\in \R\otimes_\Z J_f(\Q).$$ Write $$\pair{R}{R'} := \frac{\hat{h}(R + R') - \hat{h}(R) - \hat{h}(R')}{2},$$ so that $$\cos{\theta_{R,R'}} = \frac{\pair{R}{R'}}{|R||R'|}.$$ Let $R\in \III_f^{?,(Q,k),(i_1,\ldots,i_4),\rho,[[i]]}$ with $|v_R|$ minimal. (Of course if $\III_f^{?,(Q,k),(i_1,\ldots,i_4),\rho,[[i]]}$ is empty we are already done.) The claim is that, for all $R\neq P\in \III_f^{?,(Q,k),(i_1,\ldots,i_4),\rho,[[i]]}$, we have that $|v_P|\gg 1$. Either $|v_R|\geq \frac{1}{2} - \delta^{\frac{1}{2}}$, in which case we are done, or not. If both $|v_R|, |v_P| < \frac{1}{2} - \delta^{\frac{1}{2}}$, then $$|v_P-v_R|\leq |v_P| + |v_R| < 1 - 2\delta^{\frac{1}{2}}.$$ But we also have that $$|v_P - v_R|^2 = \left|\frac{P}{|P|} - \frac{R}{|R|}\right|^2 = 2 - 2\cos{\theta_{P,R}}.$$ However, since both $P$ and $R$ are in $\III_f^{?,(Q,k),(i_1,\ldots,i_4),\rho,[[i]]}$, by Lemmas \ref{mumford gap principle for big large points} and \ref{mumford gap principle for normal large points}, we find that $$\cos{\theta_{P,R}}\leq \frac{1}{2} + O(\delta),$$ whence $$|v_P - v_R|^2\geq 1 - O(\delta),$$ a contradiction. So, on removing $R$ from $\III_f^{?,(Q,k),(i_1,\ldots,i_4),\rho,[[i]]}$, we have that all remaining $|v_P|\gg 1$.

Now observe that, for $P\neq P'\in \III_f^{?,(Q,k),(i_1,\ldots,i_4),\rho,[[i]]} - \{R\}$,
\begin{align*}
|v_P||v_{P'}|\cos{\theta_{v_P,v_{P'}}} &= \pair{v_P}{v_{P'}}
\\&= 1 - \cos{\theta_{P,Q}} - \cos{\theta_{P',Q}} + \cos{\theta_{P,P'}}
\\&< \cos{\theta_{P,P'}} - \frac{1}{2} - 2\delta^{\frac{1}{2}}.
\end{align*}

But since $P\neq -P'$ (else the following claim is trivial anyway) and both are in $\III_f^{?,(Q,k),(i_1,\ldots,i_4),\rho,[[i]]}$, again by Lemmas 	\ref{mumford gap principle for big large points} and \ref{mumford gap principle for normal large points}, we find that $$\cos{\theta_{P,P'}}\leq \frac{1}{2} + O(\delta).$$ Thus $$|v_P||v_{P'}|\cos{\theta_{v_P,v_{P'}}} < -2\delta^{\frac{1}{2}} + O(\delta).$$ Since we have already established that $|v_P|, |v_{P'}|\gg 1$ (and of course $\delta\ll 1$), this gives the claim that $\cos{\theta_{P,P'}}\leq -\Omega(\delta^{\frac{1}{2}})$.

But it then follows that:
\begin{align*}
0&\leq \left|\sum_{P\in \III_f^{?,(Q,k),(i_1,\ldots,i_4),\rho,[[i]]} - \{R\}} \frac{P}{|P|}\right|^2 \\&= \left(\#|\III_f^{?,(Q,k),(i_1,\ldots,i_4),\rho,[[i]]} - \{R\}|\right) + \sum_{P\neq P'\in \III_f^{?,(Q,k),(i_1,\ldots,i_4),\rho,[[i]]} - \{R\}} \cos{\theta_{P,P'}}\\&\leq \left(\#|\III_f^{?,(Q,k),(i_1,\ldots,i_4),\rho,[[i]]} - \{R\}|\right) - \Omega(\delta^{\frac{1}{2}})\cdot \left(\#|\III_f^{?,(Q,k),(i_1,\ldots,i_4),\rho,[[i]]} - \{R\}|\right)^2.
\end{align*}
Rearranging now gives the result.
\end{proof}

\subsection{Conclusion of the proof.}
Thus we have completed the proof of Theorem \ref{the big kahuna}. Let us combine the ingredients to conclude.
\begin{proof}[Proof of Theorem \ref{the big kahuna}.]
By Lemmas \ref{big lookin small points} and \ref{the truly little guys}, we have seen that $\Avg(\#|\I_f|) = 0$. But, by Lemmas \ref{the medium point bound} and \ref{the large point bound}, we have also seen that $$\#|\II_f\cup \III_f|\ll 1.888^{\rank(J_f(\Q))}\leq 2^{\rank(J_f(\Q))}\leq \#|\Sel_2(J_f)|.$$ Finally, by Theorem \ref{bhargava gross}, $\Avg(\#|\Sel_2(J_f)|)\ll 1$. This concludes the proof!
\end{proof}

\section{Optimizing the bound on the number of large points}

We quickly note that, by optimizing the argument for large points, one gets the following (albeit with an uglier proof).

\begin{prop}\label{the optimized large point bound}
$$\#|\III_f|\ll 1.872^{\rank(J_f(\Q))}.$$
\end{prop}

\begin{proof}
The argument is exactly the same as in Lemma \ref{the large point bound} --- and, indeed, explains the reason for the precision in the estimate $|v_P|\geq \frac{1}{2} - \delta^{\frac{1}{2}}$ for all but at most one $P\in \III_f^{(Q,k)}$. The point is that one instead gets that, for all $P\neq P'\in \III_f^{(Q,k)} - \{R\}$, $\cos{\theta_{v_P,v_{P'}}}\leq \frac{\frac{3}{2} - 2\alpha}{|v_P||v_{P'}|} + O(\delta^{\frac{1}{2}})\leq 6 - 8\alpha + O(\delta^{\frac{1}{2}})$. The optimal choice for $\alpha$ --- keeping in mind that $\alpha > \frac{1}{\sqrt{2}}$ (and this is a crucial threshold in Vojta's method) --- turns out to be $\alpha\sim 0.7406$! Now one uses Kabatiansky-Levenshtein to bound both $S$ (where the repulsion bound is $\cos{\theta}\leq \alpha + O(\delta^{\frac{1}{2}})$, which results in a bound of $\#|S|\leq 1.85149^{\rank(J_f(\Q))}$) and $\III_f^{(Q,k)}$ (where the repulsion bound is $\cos{\theta}\leq 6 - 8\alpha + O(\delta^{\frac{1}{2}})$, which results in a bound of $\#|\III_f^{(Q,k)}|\ll 1.01077^{\rank(J_f(\Q))}$), and multiplies these bounds together to conclude.
\end{proof}

We further note that nowhere in our large point bounding did we use that we are in the special case of genus $2$ --- or even hyperelliptic --- curves over $\Q$. In general one has the Mumford gap principle bound $\cos{\theta_{P,Q}}\leq \frac{1}{g} + O(\delta)$ for $P\neq Q\in C(\Q)$ with $h(Q) = h(P)\cdot (1 + O(\delta))\gg_g h(C)$, where we define $h(C)$ to e.g.\ be the height of the equations of the tricanonically embedded $C\subseteq \P^{5g-6}$. Moreover, Vojta's theorem (with Bombieri-Granville-Pintz's determination of implicit constants) still applies in this case. These being the only ingredients we used, we may quickly prove:

\begin{prop}\label{the optimized general large point bound}
Let $C/\Q$ be a smooth projective curve of genus $g\geq 2$. Let $\max\left(\frac{1}{\sqrt{g}}, \frac{1}{4} + \frac{3}{4g}\right) < \alpha < \frac{1}{2} + \frac{1}{2g}$. Let $\eps > 0$. Then: $$\#|\{P\in C(\Q) : h(P)\gg_g h(C)\}|\ll M\left(\rank(J_f(\Q)), \alpha + \eps\right)\cdot M\left(\rank(J_f(\Q)), \frac{1 + \frac{1}{g} - 2\alpha}{\frac{1}{2} - \frac{1}{2g}} + \eps\right),$$ where $$M(n,\eta) := \max \{\#|S| : S\subseteq S^{n-1}, \forall v\neq w\in S, \cos{\theta_{v,w}}\leq \eta\}\}.$$
\end{prop}

Note that, once $g\gg 1$, we may take $\alpha\sim 0.4818$ and apply Kabatiansky-Levenshtein, obtaining the upper bound $1.311^{\rank(J_f(\Q))}$ claimed in the abstract. (See the attached Mathematica document for the optimization.)

\begin{proof}
The argument is the same as in Lemma \ref{the optimized large point bound}. We simply note that, for general $g\geq 2$, we get, for all but at most one $P\in \III_f^{(Q,k)}$, $|v_P|\geq \sqrt{\frac{1}{2} - \frac{1}{2g}} + O(\delta^{\frac{1}{2}})$, and so $\cos{\theta_{v_P,v_{P'}}}\leq \frac{1 + \frac{1}{g} - 2\alpha}{\frac{1}{2} - \frac{1}{2g}} + O(\delta^{\frac{1}{2}})$. Note also that the condition that $\alpha > \frac{1}{4} + \frac{3}{4g}$ is to ensure that this latter upper bound is at most $1$. This completes the argument.
\end{proof}

We leave the question of using these methods to bound $\#|C_f(\Q)|$ for \emph{each} $f$ to a forthcoming paper.

\bibliography{theaveragenumberofrationalpointsongenustwocurvesisbounded}{}
\bibliographystyle{plain}

\ \\

\end{document}